\documentclass[reqno,fleqn]{amsart}
\usepackage{amsfonts,amsthm,amssymb}
\usepackage{lipsum}
\usepackage[utf8]{inputenc}
\usepackage[english]{babel}

\usepackage{lmodern,textcomp}
\usepackage{amsmath}
\usepackage{mathtools}
\usepackage{latexsym}
\usepackage{tikz}
\usepackage{fancyvrb}
\usepackage{epsfig}
\usepackage{pstricks,slashbox,multirow}
\usepackage{graphicx}
\usepackage{rotating}
\usetikzlibrary{positioning}
\usepackage{subcaption}
\usepackage{datetime}
\newtheorem{theorem}{Theorem}

\newtheorem{lemma}[theorem]{Lemma}

\newtheorem{cor}[theorem]{Corollary}
\newtheorem{definition}[theorem]{Definition}

\newtheorem{prm}{Problem}
\newtheorem{oprm}{Open Problem}

\newtheorem{rem}[theorem]{Remark}

\title[A study on Type-2 isomorphic $C_n(R)$: Part 6: Abelian groups $(V_{n,m}(C_n(R)), \circ)$ and $(T2_{n,m}(C_n(R)), \circ)$]{A study on Type-2 isomorphic circulant graphs. \\ Part 6: Abelian groups $(T2_{n,m}(C_n(R)), \circ)$ and $(V_{n,m}(C_n(R)), \circ)$}

\author{\sc Vilfred Kamalappan} 
\address{Department of Mathematics, Central University of Kerala, Periye, Kasaragod, Kerala, India - 671 316.}
\email{vilfredkamal@gmail.com}  



\date{}

\begin{document} 

\begin{abstract} This article is the $6^{th}$ part of a detailed study on Type-2 isomorphic circulant graphs having ten parts \cite{v2-1}-\cite{v2-10}. In this part, we define $V_{n,m}(C_n(R))$ and Type-2 set $T2_{n,m}(C_n(R))$ of $C_n(R)$ and present their properties. We prove that $(V_{n,m}(C_n(R)), \circ)$ is an Abelian group and $(T2_{n,m}(C_n(R)), \circ)$ is a subgroup of $(V_{n,m}(C_n(R)), \circ)$ where $T2_{n,m}(C_n(R))$ =  $\{C_n(R)\}$ $\cup$ $\{C_n(S):$ $C_n(S)$ is Typ-2 isomorphic to $C_n(R)$ w.r.t. $m \}$ and $(T2_{n,m}(C_n(R)), \circ)$ is the Type-2 group of $C_n(R)$ w.r.t. $m$. We also present many examples of Type-1 and Type-2 groups where $T1_{n}(C_n(R))$ = $\{C_n(xR): x\in\varphi_{n}\}$ is the Type-1 set of $C_n(R)$ and $(T1_{n}(C_n(R)), \circ')$ is its Type-1 group. 
\end{abstract}

\maketitle

	
\section{Introduction}

This study is the $6^{th}$ part of a detailed study on Type-2 isomorphic circulant graphs by Vilfred and Wilson \cite{v2-1}-\cite{v2-10} having ten parts. We call Adam's isomorphism as Type-1 isomorphism \cite{ad67}. Vilfred \cite{v17}-\cite{v20} defined set $T1_{n}(C_n(R))$ = $Ad_{n}(C_n(R))$ = $\{T1_{n,x}(C_n(R))$ = $\varphi_{n,x}(C_n(R))$ = $C_n(xR): x\in\varphi_n\}$ and proved that $(T1_{n}(C_n(R)), \circ')$ is an Abelian group under composition of mappings. He also defined Type-2 isomorphism, different from Type-1 isomorphism, of circulant graphs and developed its theory \cite{v2-1}. In this paper, we define $V_{n,m}(C_n(R))$ and Type-2 set $T2_{n,m}(C_n(R))$ of $C_n(R)$ and present their properties. We prove that $(V_{n,m}(C_n(R)), \circ)$ is an Abelian group and $(T2_{n,m}(C_n(R)), \circ)$ is a subgroup of $(V_{n,m}(C_n(R)), \circ)$ where $T2_{n,m}(C_n(R))$ =  $\{C_n(R)\}$ $\cup$ $\{C_n(S):$ $C_n(S)$ is Typ-2 isomorphic to $C_n(R)$ w.r.t. $m \}$ and $(T2_{n,m}(C_n(R)), \circ)$ is the Type-2 group of $C_n(R)$ w.r.t. $m$. We also present many examples of Type-1 and Type-2 groups where $T1_{n}(C_n(R))$ = $\{C_n(xR): x\in\varphi_{n}\}$ is the Type-1 set of $C_n(R)$ and $(T1_{n}(C_n(R)), \circ')$ is its  Type-1 group. For basic definitions and results on isomorphic circulant graphs refer  \cite{v24, v2-1}. 

\section{Preliminaries}

We present here a few definitions and results that are required in the subsequent sections.

\begin{theorem}{\rm\cite{v17}  \quad \label{t1}  In $C_n(R)$, the length of a cycle of period $r$ is $\frac{n}{\gcd(n,r)}$ and the number of disjoint periodic cycles of period $r$ is $\gcd(n,r)$, $r \in R$. \hfill $\Box$}
\end{theorem}

\begin{cor}{\rm \cite{v17} \quad \label{c2}  In $C_n(R)$, length of a cycle of period $r$ is $n$ if and only if $\gcd(n,r)$ = $1$, $r \in R$. \hfill $\Box$}
\end{cor}

\begin{theorem}{\rm \cite{v17} \quad \label{t3} If $C_n(R)$ and $C_n(S)$ are isomorphic, then there exists a bijection $f$ from $R \to S$ such that $\gcd(n, r)$ = $\gcd(n, f(r))$ for all $r\in R$.  \hfill $\Box$  }
\end{theorem}

\section{Adam's Isomorphism or Type-1 Isomorphism and Type-1 set of $C_n(R)$}

In this section, we present our study on Adam's isomorphism or Type-1 isomorphism of circulant graph $C_n(R)$ and define Type-1 set of $C_n(R)$. 

\begin{lemma}{\rm \cite{v17} \quad \label{l4} Let $S \subseteq \mathbb{Z}_n$, $S \neq \emptyset$ and $x \in \mathbb{Z}_n$. Define a mapping $\varphi_{n, x}:$ $S$ $\rightarrow$ $\mathbb{Z}_n$ $\ni$ $\varphi_{n, x}(s)$ = $xs$, $\forall$ $s \in S,$ under multiplication modulo $n$. Then $\varphi_{n, x}$ is bijective if and only if $S = \mathbb{Z}_n$ and $\gcd(n, x) = 1$. }
\end{lemma}
\begin{proof}\quad Here, we give a proof using a property of periodic cycles of a circulant graph $C_n(R)$ with jump size $x$. Let $S = \mathbb{Z}_n$. Then the numbers $0,$ $x,$ $2x,$ $3x,$ $\dots,$ $(n-1)x$, under multiplication modulo $n$, are all distinct if and only if $\gcd(n, x)$ = 1 since the cycle of period $x$ in $C_n(R),$ $x \in R$, is of length $n$ if and only if $\gcd(n, x)$ = 1, using Corollary \ref{a2}.
	
Conversely, if $S \neq \mathbb{Z}_n$, then $S$ is a proper subset of $\mathbb{Z}_n$ and so $\varphi_{n,x}$ is not a bijective mapping. Hence the result follows. 
\end{proof} 

Hereafter, unless otherwise it is mentioned in other way, we consider $\varphi_{n,x}$ with $\gcd(n,x)$ = 1 only. Let $\varphi_n$ = $\{ x \in \mathbb{Z}_n : \gcd(n, x) = 1 \}$. Clearly,  $( \varphi_n,~\circ)$ is an abelian group under the binary operation $\lq\circ\rq$,  multiplication modulo $n$.

\begin{definition}{\rm\cite{ad67}} \quad \label{d5} For $R =$ $\{r_1$, $r_2$, $\dots$, $r_k\}$ and $S$ = $\{s_1$, $s_2$, $\dots$, $s_k\}$, circulant graphs $C_n(R)$ and $C_n(S)$ are {\em Adam's isomorphic} if there exists a positive integer $x$ $\ni$ $\gcd(n, x)$ = 1 and $S$ = $\{xr_1$, $xr_2$, $\dots$, $xr_k\}_n^*$ where $<r_i>_n^*$, the {\it reflexive modular reduction} of a sequence $< r_i >$, is the sequence obtained by reducing each $r_i$ under modulo $n$ to yield $r_i'$ and then replacing all resulting terms $r_i'$ which are larger than $\frac{n}{2}$ by $n-r_i'.$  
\end{definition}

The following definition based on Lemma \ref{l4} is same as the above.

\begin{definition} {\rm \cite{v17}} \label{d6} Let $n\in\mathbb{N}$. For a subset $S$ of $\mathbb{Z}_n$ one can define the circulant graph $C_n(S)$ as the graph with vertex set $\mathbb{Z}_n$ and edge set $\{\{x, y\} | x,y\in \mathbb{Z}_n,$ $x-y\in S\}$. For an integer $k$ coprime to $n$, there is a very natural isomorphism between circulants $C_n(S)$ and $C_n(kS)$  namely the map $\varphi_k : x \mapsto kx$ (observe that $\varphi_k$ is an automorphism of the group $\mathbb{Z}_n$) where $kS$ = $\{ks : s\in S\}$. In this case, we say that $C_n(S)$ and $C_n(kS)$ are \emph{Adam's isomorphic} or \emph{Type-1 isomorphic} and we call the isomorphism as \emph{Type-1 isomorphism}.  
\end{definition}

\begin{definition}{\rm \cite{v17}} \label{d7} Let $Ad_n = \{\varphi_{n,x}: x\in \varphi_n\}$, $Ad_n(S) = \{\varphi_{n,x}(S): x\in \varphi_n\}$ = $\{xS: x\in \varphi_n\}$, $Ad_{n,x}(C_n(R))$ = $T1_{n,x}(C_n(R))$ = $\varphi_{n,x}(C_n(R))$ = $C_n(\varphi_{n,x}(R))$ = $C_n(xR)$, $x\in \varphi_n$ and $Ad_n(C_n(R)) = T1_n(C_n(R)) = \{\varphi_{n,y}(C_n(R)) = C_n(yR): y\in \varphi_n\}$ for sets $R,S \subseteq \mathbb{Z}_n$ where $\varphi_{n,x}(R)$ in $C_n(\varphi_{n,x}(R))$ is calculated under the reflexive modulo $n$. We call the set $T1_n(C_n(R))$ as {\em Type-1 set} of $C_n(R)$.  Define $``\circ'"$ in $Ad_n(C_n(R))$ such that $\varphi_{n,x} \circ' \varphi_{n,y}$ = $\varphi_{n,xy}$, $C_n(xR) \circ'  C_n(yR)$ = $C_n((xy)R)$ and $(\varphi_{n,x} \circ' \varphi_{n,y})(C_n(R))$ = $\varphi_{n,x}(C_n(R)) \circ' \varphi_{n,y}(C_n(R))$ $\forall$ $x,y \in \varphi_n$, under arithmetic modulo $n$. 
\end{definition}

Clearly, $Ad_n(C_n(R))$ is the set of all circulant graphs which are Type-1 isomorphic to $C_n(R)$ and $(Ad_n(C_n(R)), \circ' )$ = $(T1_n(C_n(R)), \circ' )$ is an Abelian group and we call it as the {\em Adam's group} or {\em Type-1 group of} $C_n(R)$ under $``\circ'"$. Moreover, $(\varphi_{n,x} \circ' \varphi_{n,y})(C_n(R))$ = $\varphi_{n,x}( \varphi_{n,y}(C_n(R)))$ = $(\varphi_{n,x}(C_n(yR)))$ = $C_n(x(yR))$ = $C_n((xy)R)$ = $C_n(xR)$ $\circ'$  $C_n(yR)$ = $\varphi_{n,x}(C_n(R))$ $\circ'$ $\varphi_{n,y}(C_n(R))$, $\forall$ $x,y\in \varphi_n$, $(xy)R$, $xR$ and $yR$ are calculated under reflexive modulo $n$ and $xy$ is calcuated under multiplication modulo $n$.

\begin{theorem} \quad \label{t8} {\rm Let $Ad_n(C_n(R))$ = $\{\varphi_{n,x}(C_n(R)) = C_n(xR): x\in\varphi_n \}$. Then, $C_n(S)\in Ad_n(C_n(R))$ if and only if $Ad_n(C_n(R))$ = $Ad_n(C_n(S))$ if and only if $C_n(R)\in Ad_n(C_n(S))$. }
\end{theorem}
\begin{proof}\quad Let $C_n(S)\in Ad_n(C_n(R))$. This implies, $\exists$ $x\in\varphi_n$ $\ni$ $C_n(S)$ = $\varphi_{n,x}(C_n(R))$ = $C_n(xR)$. We have, corresponding to $x\in \varphi_n$, $\exists$ $x^*\in \varphi_n$ $\ni$ $xx^*$ = $1\in\varphi_n$ and $\varphi_{n,x^*}(C_n(S))$ = $\varphi_{n,x^*}(C_n(xR))$ = $C_n(x^*(xR))$ = $C_n((x^*x)R)$ = $C_n(R)$. Thus, $C_n(R)$ = $\varphi_{n,x^*}(C_n(S))$ which implies, $C_n(R)\in Ad_n(C_n(S))$, $x^*\in \varphi_n$. Similarly, we can show that if $C_n(R)\in Ad_n(C_n(S))$, then $C_n(S)\in Ad_n(C_n(R))$.

Also, let $C_n(T)\in Ad_n(C_n(R))$ for some $T$. This implies, $\exists$ $y\in\varphi_n$ $\ni$ $C_n(T)$ = $\varphi_{n,y}(C_n(R))$ = $\varphi_{n,y}(\varphi_{n,x^*}(C_n(S)))$ = $\varphi_{n,y}(C_n(x^*S))$ = $C_n(y(x^*S))$ = $C_n((yx^*)S)$, $y,x^*,yx^*\in\varphi_n$. This implies, $C_n(T)\in Ad_n(C_n(S))$. This implies that $Ad_n(C_n(R))$ $\subseteq$ $Ad_n(C_n(S))$. Similarly, we can show that $Ad_n(C_n(S))$ $\subseteq$ $Ad_n(C_n(R))$. 

This implies, $Ad_n(C_n(R))$ = $Ad_n(C_n(S))$. Hence, we get the result.
\end{proof}

{\color{red} From (a2) of problem 1} in \cite{v2-4}, we get, $T1_{54}(C_{54}(1,17,18,19))$ = $\{\varphi_{54,x}(C_{54}(1,17,18,19)): x\in \varphi_{54}\}$ = $\{\varphi_{54,x}(C_{54}(1,17,18,19)): x = 1,5,7,11,13,17,19,23,25,29,31,35,37,41,43,47,49,53\}$ = $\{C_{54}(1,17$, $18,19),$ $C_{54}(5,13,18,23)$, $C_{54}(7,11,18,25)\}$ = $T1_{54}(C_{54}(5,13,18,23))$ = $T1_{54}(C_{54}(7,11,18,25))$. 

\medskip
\noindent
{\bf Notations.} \label{2.6} We introduce the following notations under Type-1 isomorphism of circulant graphs.

$C_n(R) \cong_{Ad_{n,x}} C_{n}(S)$ or $C_n(R) \cong_{T1_{n,x}} C_{n}(S)$ when 
$C_n(R)$ and $C_n(S)$ are Type-1 isomorphic and $Ad_{n,x}(C_n(R))$ (= $T1_{n,x}(C_n(R))$ = $\varphi_{n,x}(C_n(R))$ = $C_{n}(xR)$) = $C_n(S)$, $x\in\varphi_n$; and $C_n(R) \cong_{Ad_{n}} C_{n}(S)$ or $C_n(R) \cong_{T1_{n}} C_{n}(S)$ when $C_n(R)$ and $C_n(S)$ are Adam's isomorphic or Type-1 isomorphic. 

\section{Known results on isomorphic circulant graphs of Type-2 w.r.t. $m$ = 2,3,5,7} 

In 1996, Vilfred \cite{v96} defined Type-2 isomorphism of $C_n(R)$ w.r.t. $m$ $\ni$ $m$ = $\gcd(n, r) > 1$, $r\in R$ and $r,n\in\mathbb{N}$ and studied Type-2 isomorphic circulant graphs w.r.t. $m$ = 2  in \cite{v13,v20}. And with Wilson \cite{v24} obtained families of isomorphic circulant graphs of Type-2 w.r.t. $m$ = 3,5,7. In \cite{v2-2-arX}, Vilfred modified the definition of Type-2 isomorphism of circulant graphs $C_n(R)$ w.r.t. $m$ by considering $m > 1$ as a divisor of $\gcd(n, r)$ and $r\in R$ and in \cite{v2-1} and he further modified the definition to $r\in R$ and $m > 1$ and $m^3$ are divisors of $\gcd(n, r)$ and $n$, respectively. In this section, we present known results on isomorphic circulant graphs of Type-2 w.r.t. $m$ = 2,3,5,7.

\begin{lemma}{\rm \cite{v2-1} \quad \label{a7}  Let $m >1$ be a divisor of $n$. Then, for each value of $t$, the mapping $\theta_{n,m,t}: \mathbb{Z}_n \rightarrow \mathbb{Z}_n$ defined by $\theta_{n,m,t}(x)$ = $x+jtm$, $x \in \mathbb{Z}_n$, is bijective under arithmetic modulo $n$ where $x = qm+j,$ $0 \leq j \leq m-1,$ $0 \leq q,t \leq \frac{n}{m}-1$ and $j,m,q,t \in \mathbb{Z}_n$. In particular, the result is true when $m$ = $\gcd(n, r) > 1$ and $r \in \mathbb{Z}_n$. \hfill $\Box$} 
\end{lemma}

\begin{definition} \cite{v2-1}\quad  \label{d4.2} Let $V(K_n) = \{u_0,u_1,u_2,...,u_{n-1}\}$, $V(C_n(R)) = \{v_0,v_1,v_2,...,$ $v_{n-1}\}$, $r\in R$, $|R| \geq 3$, $m > 1$ and $m$ and $m^3$ be divisors of $\gcd(n, r)$ and $n$, respectively.  Define one-to-one mapping $\theta_{n,m,t} :$ $V(C_n(R)) \rightarrow V(K_n)$ such that $\theta_{n,m,t}(v_x)$ = $u_{x+jtm}$,  $\theta_{n,m,t}((v_x, v_{x+s}))$ = $(\theta_{n,m,t}(v_x), \theta_{n,m,t}(v_{x+s}))$ under subscript arithmetic modulo $n$ and $\theta_{n,m,t}(C_n(R))$ = $C_n(\theta_{n,m,t}(R))$ for every $x$ = $qm+j \in \mathbb{Z}_n$, $s\in R$, $0 \leq j \leq m-1$, $0 \leq q,t \leq \frac{n}{m} -1$ and $\theta_{n,m,t}(R)$ in $C_n(\theta_{n,m,t}(R))$ is calculated under the reflexive modulo $n$. And for a particular value of $t,$ if  $\theta_{n,m,t}(C_n(R))$ = $C_n(S)$ for some $S$  and  $S \neq yR$ for all $y\in \varphi_n$ under reflexive modulo $n$, then $C_n(R)$ and $C_n(S)$ are called {\em isomorphic circulant graphs of Type-2 w.r.t. $m$} and the isomorphism as {\em Type-2 isomorphism w.r.t. $m$.} 

When $C_n(R)$ and $C_n(S)$ are Type-2 isomorphic w.r.t. $m$, then we also say that $C_{kn}(kR)$  and $C_{kn}(kS)$ are Type-2 isomorphic w.r.t. $m$, $k\in\mathbb{N}$. Here, $k.C_n(T)$ = $C_{kn}(kT)$, $k\in\mathbb{N}$.
\end{definition}

\begin{rem}\quad \label{r11} Following steps are used to establish isomorphism of Type-2 w.r.t. $m$ between circulant graphs $C_n(R)$ and $C_n(S)$. (i) $R \neq S$ and $|R|$ = $|S| \geq 3$; (ii) $\exists$ $r\in R,S$, $m > 1$ $\ni$ $m$ and $m^3$ are divisors of $\gcd(n, r)$ and $n$, respectively, and for some $t$ $\ni$ $1 \leq t \leq \frac{n}{m} -1$, $\theta_{n,m,t}(C_n(R))$ = $C_n(S)$ and (iii) $S \neq xR$ for all $x\in\varphi_n$ under arithmetic reflexive modulo $n$. 
\end{rem} 

\begin{rem} \label{r12} \quad While searching for possible value(s) of $t$ for which the transformed graph $\theta_{n,m,t}(C_n(R))$ is circulant of the form $C_n(S)$ for some $S \subseteq [1, \frac{n}{2}]$,  calculation on $r_i$s which are integer multiples of $m$ need not be done  under the transformation $\theta_{n,m,t}$ as there is no change in these $r_i$s where $m > 1$, $m$ and $m^3$ are divisors of $\gcd(n, r)$ and $n$, respectively, and $r\in R$. Also, for a given circulant graph $C_n(R)$, w.r.t. different values of $m$, we may get different Type-2 isomorphic circulant graphs.
\end{rem}

\begin{theorem}{\rm \cite{v2-1}}\quad \label{t13} {\rm For $n \geq 2$, $1 \leq 2s-1 \leq 2n-1$, $n \neq 2s-1$, $R$ = $\{2,2s-1, 4n-(2s-1)\}$ and $S$ = $\{ 2,$ $2n-(2s-1)$, $2n+2s-1 \}$, $\theta_{8n,2,n}(C_{8n}(R))$ = $C_{8n}(S)$ = $\theta_{8n,2,3n}(C_{8n}(R)),$ $\theta_{8n,2,n}(C_{8n}(S))$ = $C_{8n}(R)$ = $\theta_{8n,2,3n}(C_{8n}(S))$ and circulant graphs $C_{8n}(R)$ and $C_{8n}(S)$ are Type-2 isomorphic  w.r.t. $m$ = 2. When $n$ = $2s-1$, the two circulant graphs are the same. \hfill $\Box$}
\end{theorem}

\begin{theorem} {\rm \cite{v2-1}} \label{t14} {\rm Let $n \geq 2$, $k \geq 3$, $1 \leq 2s-1 \leq 2n-1$, $n \neq 2s-1$, $R$ = $\{ 2s-1,$ $4n-(2s-1),$ $2p_1,$ $2p_2,$ $\dots,$ $2p_{k-2} \}$, $S$ = $\{2n-(2s-1),$ $2n+2s-1,$ $2p_1,2p_2,\dots,2p_{k-2}\}$, $2y\in R,S$, $\gcd(4n,y)$ = 1, $p_1,p_2,\dots,p_{k-2} \in \mathbb{N}$ and $\gcd(p_1,p_2,\dots,p_{k-2})$ = 1. Then, (i) $\theta_{8n,2,n}(C_{8n}(R))$ = $C_{8n}(S)$ = $\theta_{8n,2,3n}(C_{8n}(R))$, $\theta_{8n,2,n}(C_{8n}(S))$ = $C_{8n}(R)$ = $\theta_{8n,2,3n}(C_{8n}(S))$, $\theta_{8n,2,2n}(C_{8n}(R))$ = $C_{8n}(R)$, $\theta_{8n,2,2n}($ $C_{8n}(S))$ = $C_{8n}(S)$ and (ii) for given values of $n,s,p_1,p_2,\dots,p_{k-2}$ and $y$, $C_{8n}(R)$ and $C_{8n}(S)$ are isomorphic of either Type-1 or Type-2 w.r.t. $m$ = 2. Moreover, for given $n, s$ and $k$ and for all such possible values of $p_1,p_2,\dots,p_{k-2}$ and $y$, the set $\{ C_n(S)$ = $\theta_{8n,2,n}(C_{8n}(R)):$  $\gcd(p_1,p_2,\dots,p_{k-2})$ = 1, $p_1,p_2,\dots,p_{k-2} \in \mathbb{N}\}$ contains all isomorphic circulant graphs of $C_{8n}(R)$ of Type-2 w.r.t. $m$ = 2.      \hfill $\Box$}
\end{theorem}

\begin{theorem}{\rm \cite{v24}}\quad \label{t15} {\rm For $n \geq 2$, $1 \leq 2s-1 < 2s'-1 \leq [\frac{n}{2}]$, $0 \leq t \leq [\frac{n}{2}]$, $R$ = $\{2,2s-1, 2s'-1\}$ and $n,s,s'\in \mathbb{N}$, if $\theta_{n,2,t}(C_n(R))$ and $C_n(R)$ are  isomorphic circulant graphs of Type-2 w.r.t. $m$ = 2 for some $t$, then $n \equiv 0~(mod ~ 8)$, $2s-1+2s'-1$ = $\frac{n}{2}$, $2s-1 \neq \frac{n}{8}$, $t$ = $\frac{n}{8}$ or $\frac{3n}{8}$, $1 \leq 2s-1 \leq \frac{n}{4}$ and $n \geq 16$. In particular, when $R$ = $\{2, 2s-1, 4n-(2s-1)\}$, $S$ = $\{2, 2n-(2s-1), 2n+2s-1\}$, $n\geq 2$ and $n,s\in \mathbb{N}$, $\theta_{8n,2,n}(C_{8n}(R))$ = $C_{8n}(S)$ = $\theta_{8n,2,3n}(C_{8n}(R))$, $\theta_{8n,2,n}(C_{8n}(S))$ = $C_{8n}(R)$ = $\theta_{8n,2,3n}(C_{8n}(S))$, $\theta_{8n,2,2n}(C_{8n}(R))$ = $C_{8n}(R)$, $\theta_{8n,2,2n}(C_{8n}(S))$ = $C_{8n}(S)$ and $C_{8n}(R)$ and $C_{8n}(S)$ are Type-2 isomorphic w.r.t. $m$ = 2.  \hfill $\Box$}
\end{theorem}

\begin{theorem} \cite{v24} \label{t16} {\rm For $R$ = $\{1, 3, 9n-1, 9n+1\}$, $S$ = $\{3, 3n+1, 6n-1, 12n+1\}$, $T$ = $\{3, 3n-1, 6n+1$, $12n-1\}$ and $n\in\mathbb{N}$, $\theta_{27n,3,n}(C_{27n}(R))$ = $C_{27n}(S)$, $\theta_{27n,3,n}(C_{27n}(S))$ = $C_{27n}(T)$, $\theta_{27n,3,n}(C_{27n}(T))$ = $C_{27n}(R)$ and $C_{27n}(R)$, $C_{27n}(S)$ and $C_{27n}(T)$ are Type-2 isomorphic circulant graphs w.r.t. $r$ = 3. \hfill $\Box$}
\end{theorem}

\begin{theorem} \label{t17} {\rm Let $k \geq 3$, $R$ = $\{1, 9n-1, 9n+1, 3p_1, 3p_2, \dots, 3p_{k-2}\}$, $S$ = $\{3n+1, 6n-1, 12n+1,$ $3p_1, 3p_2, \dots, 3p_{k-2}\}$, $T$ = $\{3n-1, 6n+1, 12n-1, 3p_1, 3p_2, \dots, 3p_{k-2}\}$, $\gcd(p_1,p_2,\dots,p_{k-2}) = 1$ and $k,n,p_1,p_2,\dots,p_{k-2}\in\mathbb{N}$. Then, $(i)$  $\theta_{27n,3,n}(C_{27n}(R))$ = $C_{27n}(S)$, $\theta_{27n,3,n}(C_{27n}(S))$ = $C_{27n}(T)$ and $\theta_{27n,3,n}(C_{27n}(T))$ = $C_{27n}(R)$ and $(ii)$ for given values of $k,p_1,p_2,\dots,p_{k-2}$ and $n$, $C_{27n}(R)$, $C_{27n}(S)$ and $C_{27n}(T)$ are isomorphic and are either Type-1 or Type-2 w.r.t. $m$ = 3.  \hfill $\Box$}
\end{theorem}

\begin{theorem} \cite{v24} \label{t18} {\rm For $R_i$ = $\{5, d_i, 25n-d_i, 25n+d_i, 50n-d_i, 50n+d_i\}$, $d_i$ = $5n(i-1)+1$, $i,j$ = 1 to 5 and $n\in\mathbb{N}$, $\theta_{125n,5,jn}(C_{125n}(R_i))$ = $C_{125n}(R_{i+j})$ where $i+j$ in $R_{i+j}$ is calculated under addition modulo 5 and $C_{125n}(R_i)$ are Type-2 isomorphc circulant graphs w.r.t. $r$ = 5.  		\hfill $\Box$}
\end{theorem}

\begin{theorem} \label{t19} {\rm Let $k \geq 3$, $d_i$ = $5n(i-1)+1$, $1 \leq i \leq 5$, $R_i$ = $\{d_i, 25n-d_i, 25n+d_i, 50n-d_i$, $50n+d_i, 5p_1, 5p_2, \dots, 5p_{k-2}\}$,  $k,n,p_1,p_2,\dots,p_{k-2}\in\mathbb{N}$ and $\gcd(p_1,p_2,\dots,p_{k-2}) = 1$. Then, for given values of $k,p_1,p_2,\dots,p_{k-2}$ and $n$, circulant graphs $C_{125n}(R_i)$ are isomorphic and are either Type-1 or Type-2 w.r.t. $m$ = 5, $1 \leq i \leq 5$.  \hfill $\Box$}
\end{theorem}

\begin{theorem} \cite{v24} \label{t20} {\rm For $R_i$ = $\{7, d_i, 49n-d_i, 49n+d_i, 98n-d_i, 98n+d_i, 147n-d_i, 147n+d_i\}$, $d_i$ = $7n(i-1)+1$, $i,j$ = 1 to 7 and $n\in\mathbb{N}$, $\theta_{343n,7,jn}(C_{343n}(R_i))$ = $C_{343n}(R_{i+j})$ where $i+j$ is calculated under addition modulo 7 and $C_{343n}(R_i)$ are Type-2 isomorphic circulant graphs w.r.t. $r$ = 7. 	\hfill $\Box$}
\end{theorem}

\begin{theorem} \label{t21} {\rm Let $k \geq 3$, $d_i$ = $7n(i-1)+1$, $1 \leq i \leq 7$, $R_i$ = $\{d_i, 49n-d_i, 49n+d_i, 98n-d_i$, $98n+d_i$, $147n-d_i,$ $147n+d_i,$ $7p_1, 7p_2, \dots, 7p_{k-2}\}$, $k,n,p_1,p_2,\dots,p_{k-2}\in\mathbb{N}$ and $\gcd(p_1,p_2,\dots,p_{k-2}) = 1$. Then for given values of $k,p_1,p_2,\dots,p_{k-2}$ and $n$, circulant graphs $C_{343n}(R_i)$ are isomorphic and are either Type-1 or Type-2 w.r.t. $m$ = 7, $1 \leq i \leq 7$.   \hfill $\Box$}
\end{theorem}

\section{On Abelian groups $(V_{n,m}(C_n(R)), \circ)$ and $(T2_{n,m}(C_n(R)), \circ)$}

In this section, based on the modified definition \ref{d4.2} for isomorphic circulant graphs of Type-2 w.r.t. $m$, we define $V_{n,m}(C_n(R))$ and $T2_{n,m}(C_n(R))$, similar to $Ad_n(C_n(R))$ = $T1_n(C_n(R))$ for any circulant graph $C_n(R)$ with $r\in R$, $m > 1$ divises $\gcd(n, r)$ and $m^3$ divides $n$. We prove that  $(V_{n,m}(C_n(R)), \circ)$ and $(T2_{n,m}(C_n(R)), \circ)$ are Abelian groups and $T2_{n,m}(C_n(R))$ $\subseteq$ $V_{n,m}(C_n(R))$. We call $(T2_{n,m}(C_n(R)), \circ)$ as {\em Type-2 group of $C_n(R)$} w.r.t. $m$ and the group $(Ad_{n}(C_n(R)), \circ')$ = $(T1_{n}(C_n(R)), \circ')$ as the {\em Type-1 group of $C_n(R)$} or {\em Adam's group of $C_n(R)$}.

\begin{definition} \quad \label{d22} Let $V(C_n(R)) = \{v_0, v_1, \dots, v_{n-1}\}$, $V(K_n)$ = $\{u_0, u_1,\dots, u_{n-1}\}$, $x = qm+j,$ $0 \leq j \leq m-1$, $m > 1$ divide $\gcd(n, r)$, $m^3$ divide $n$, $0 \leq q,t,t' \leq \frac{n}{m}-1$, $m,q,t,t',x\in \mathbb{Z}_n$ and $r \in R$. Define $\theta_{n,m,t}:$ $V(C_n(R))$ $\rightarrow$  $V(K_n)$ $\ni$ $\forall$  $x\in \mathbb{Z}_n$, $\theta_{n,m,t}(v_x)$ = $u_{x+jmt}$ and $\theta_{n,m,t}((v_x, v_{x+s}))$ = $(\theta_{n,m,t}(v_x), \theta_{n,m,t}(v_{x+s}))$  and $s\in R,$ under subscript arithmetic modulo $n$. Let $V_{n,m}$ = $\{\theta_{n,m,t}:$ $t$ = $0,1,\dots,\frac{n}{m}-1\}$, $V_{n,m}(v_s)$ = $\{\theta_{n,m,t}(v_s):$ $t = 0,1,\dots, \frac{n}{m}-1\}$, $s \in \mathbb{Z}_n$ and $V_{n,m}(C_n(R))$ = $\{\theta_{n,m,t}(C_n(R)) = C_n(\theta_{n,m,t}(R)):$ $t$ = $0,1,\dots,\frac{n}{m}-1\}$ where $\theta_{n,m,t}(R)$ in $C_n(\theta_{n,m,t}(R))$  is calculated under reflexive modulo $n$. Define $'\circ'$ in $V_{n,m} \ni \theta_{n,m,t} \circ  \theta_{n,m,t'}$ =  $\theta_{n,m,t+t'}$ and $(\theta_{n,m,t} ~\circ ~ \theta_{n,m,t'})(C_n(R))$ = $\theta_{n,m,t}(C_n(R)) ~\circ ~ \theta_{n,m,t'}(C_n(R))$ = $\theta_{n,m,t+t'}(C_n(R))$, $\forall$ $\theta_{n,m,t},\theta_{n,m,t'}\in V_{n,r}$ where $t+t'$ is calculated under arithmetic modulo ~$\frac{n}{m}$.
\end{definition}

$V_{n,m}(C_n(R))$ = $\{\theta_{n,m,t}(C_n(R)): t = 0,1,\dots, \frac{n}{m}-1\}$ and for $t$ = 0 to $\frac{n}{m}-1$, the $\frac{n}{m}$ graphs $\theta_{n,m,t}(C_n(R))$ are isomorphic and 
$V_{n,m}(C_n(R))$ contains all isomorphic circulant graphs of Type-2 of $C_n(R)$ w.r.t. $m$, if exist, under the transformation $\theta_{n,m,t}$ on $C_n(R)$ where $r\in R$, $m > 1$ divides $\gcd(n, r)$ and $m^3$ divides $n$. Now, consider an algebraic property of $V_{n,m}(C_n(R))$. 

\begin{theorem}{\rm  \quad \label{t23}  Under the above definition of $`\circ'$, $(V_{n,m}(C_n(R)), \circ)$ is an Abelian  group.}
\end{theorem}
\begin{proof}   We have, $V_{n,m}(C_n(R))$ = $\{\theta_{n,m,t}(C_n(R)): t = 0,1,\dots, \frac{n}{m}-1\}$ and for $t$ = 0 to $\frac{n}{m}-1$, the $\frac{n}{m}$ graphs $\theta_{n,m,t}(C_n(R))$ are isomorphic where $r\in R$, $m > 1$ divides $\gcd(n, r)$ and $m^3$ divides $n$.  Also, $\forall t,t'\in \{0,1,\dots,$ $\frac{n}{m}-1\}$, $\theta_{n,m,t}(C_n(R)) \circ  \theta_{n,m,t'}(C_n(R))$ = $\theta_{n,m,t+t'}(C_n(R))\in V_{n,m}$ where $t+t'$ is calculated under addition modulo $\frac{n}{m}.$ This implies, $V_{n,m}(C_n(R))$ is closed under $`\circ'$. $\theta_{n,m,0}(C_n(R))$ = $C_n(R)$ is the identity element. $`\circ'$ is abelian since $\theta_{n,m,t}(C_n(R)) \circ \theta_{n,m,t'}(C_n(R))$ = $\theta_{n,m,t+t'}(C_n(R))$ = $\theta_{n,m,t'+t}(C_n(R))$ = $\theta_{n,m,t'}(C_n(R)) \circ \theta_{n,m,t}(C_n(R))$, $\forall t,t'\in\{0,1,\dots,\frac{n}{m}-1\}$. The associative law can be proved very easily.  $\theta_{n,m,\frac{n}{m}-t}(C_n(R))$ is the inverse of $\theta_{n,m,t}(C_n(R))$, $\forall t\in \{0,1,\dots, \frac{n}{m}-1\}$ since $\theta_{n,m,\frac{n}{m}-t}(C_n(R)) \circ \theta_{n,m,t}(C_n(R))$ = $\theta_{n,m,(\frac{n}{m}-t)+t}(C_n(R))$ = $\theta_{n,m,0}(C_n(R))$. Hence the result. 
\end{proof}

\begin{theorem}  \label{t24} {\rm Let $n,r\in \mathbb{N}$, $m > 1$ divide $\gcd(n, r)$, $m^3$ divide $n$, $r \notin R$, $\theta_{n,m,t}(C_n(R \cup \{ r \}))$ = $C_n(S)$ and $C_n(R \cup \{ r \})$ and $C_n(S)$ be Type-2 isomorphic w.r.t. $m$, 	$1 \leq t \leq \frac{n}{m} - 1$. Then, there doesn't exist circulant graph that is isomorphic of Type-2 w.r.t. $m$ to $C_n(R)$. That is $C_n(R)$ and $\theta_{n, m,t}(C_n(R))$ = $C_n(S \setminus\{r\})$ are isomorphic but not of Type-2 w.r.t. $m$ where $m > 1$ divides $\gcd(n, r)$, $m^3$ divides $n$ and $0 \leq t \leq \frac{n}{m} - 1$.}
\end{theorem}
\begin{proof}\quad Let $R = \{r_1,r_2,\dots,r_k\}$, $m > 1$ divide $\gcd(n, r)$, $m^3$ divide $n$ and $\gcd(n,r_i) \neq lm$ for $i = 1,2,\dots,k$ and $l\in\mathbb{N}$.
	
	Given, $\theta_{n,m,t}(C_n(R \cup \{ r \}))$ = $C_n(S)$ $\neq$ $C_n(x(R\cup \{r\}))$, $\forall$ $x\in \varphi_n$  and $C_n(R \cup \{ r \})$ and $C_n(S)$ are Type-2 isomorphic w.r.t. $m$, $1 <x < n-1.$
	
	This implies, $\theta_{n,m,t}(C_n(R))$ $\cup$ $\theta_{n,m,t}(C_n(r))$ = $C_n(S)$ and $r\in S$ $\ni$ $m > 1$ divides $\gcd(n, r)$ and $m^3$ divides $n$. This implies, $\theta_{n,m,t}(C_n(R))$ = $C_n(S \setminus \{r\})$ and $r\notin R,S \setminus \{r\}$ \hfill $(a)$
\\
 since $\theta_{n,m,t}(r)$ = $r$ and $\theta_{n,m,t}(C_n(r))$ = $C_n(\theta_{n,m,t}(r))$ = $C_n(r)$. 
	
	Now, for every $s\in S \setminus \{r\}$, either $m > 1$ does't divide $\gcd(n, s)$ or $m^3$ is not a divisor of $n$. And so $C_n(S \setminus \{r\})$ can not be Type-2 isomorphic w.r.t. {$m$ to any} circulant graph $C_n(R)$ where $m > 1$ divides $\gcd(n, r)$ and $m^3$ divides $n$.
	
	Also, if $\theta_{n,m,t}(C_n(U))$ = $C_n(W)$ for some $t$, then 
	
	\hspace{.7cm} $\theta_{n,m,\frac{n}{m}-t}(C_n(W)) = \theta_{n,m,\frac{n}{m}-t}(\theta_{n,m,t}(C_n(U))) = \theta_{n,m,\frac{n}{m}-t+t}(C_n(U))  = C_n(U)$. \hfill $(b)$
	
	Using $(b)$ in $(a)$, we get, $\theta_{n,m,\frac{n}{m}-t}(C_n(S \setminus \{r\}))$ = $C_n(R)$ which implies that $C_n(S \setminus \{r\})$ and $C_n(R)$ are isomorphic circulant graphs but they are not of Type-2 w.r.t. $m$ since either $m$ does't divide $\gcd(n, s)$ or $m^3$ is not a divisor of $n$ for every $s\in S \setminus \{r\}$. This implies, $C_n(R)$ and $\theta_{n, m,t}(C_n(R))$ = $C_n(S \setminus\{r\})$ are isomorphic but not of Type-2 w.r.t. m, $0 \leq t \leq \frac{n}{m} - 1$.
\end{proof}

Some properties of $\theta_{n,m_i,t}(C_n(R))$ are given below. 

\vspace{.1cm} 
\noindent
{\bf Properties of $\theta_{n,m,t}(C_n(R)):$}	
\begin{enumerate}
	\item [\rm (a)]  \label{r11a}  Let $\theta_{n,m,t}(C_n(R)) = C_n(S)$, $m > 1$ divide $\gcd(n, r)$, $m^3$ divide $n$, $r\in R$ and $r_i \in [1, \frac{n}{2}]$ $\ni$ $\gcd(n, r_i)$ = $\gcd(n,r)$. Then, $r_i \in R$ if and only if $r_i \in S$ follows from the definition of $\theta_{n,m,t}$.
	
	\item [\rm (b)]  \label{r11b} For a given circulant graph $C_n(R)$ and a particular value of $t$ if $\theta_{n,m,t}(C_n(R))$ = $C_n(S)$ for some $S$, then $\forall$ $t'$ $\ni$ $0 \leq t,t' \leq \frac{n}{m} -1$, $\theta_{n,m,t+t'}(C_n(R))$ = $\theta_{n,m,t'}(C_n(S))$ since $\theta_{n,m,t+t'}(C_n(R))$ = 
	\\
 $\theta_{n,m,t'+t}(C_n(R))$ = $\theta_{n,m,t'}(\theta_{n,m,t}(C_n(R)))$ = $\theta_{n,m,t'}(C_n(S))$ where $m$ divides $\gcd(n,r)$ and $m^3$ divides $n$.
	
	\item [\rm (c)]  \label{r11c} Let $C_n(R)$ $\cong$ $C_n(S)$ and $0 \leq t \leq \frac{n}{m}-1$. Then, for some $t$, $C_n(S)$ = $\theta_{n,m,t}(C_n(R))$  if and only if $C_n(R)$ = $\theta_{n,m,\frac{n}{m}-t}(C_n(S))$. This follows from $\theta_{n,m,\frac{n}{m}-t}(C_n(S))$ = $\theta_{n,m,\frac{n}{m}-t}(\theta_{n,m,t}(C_n(R)))$ = $\theta_{n,m,(\frac{n}{m}-t)+t}(C_n(R))$ = $\theta_{n,m,\frac{n}{m}}(C_n(R))$ = $\theta_{n,m,0}(C_n(R))$ = $C_n(R)$.

\vspace{.1cm} 
Circulant graph	$C_n(R)$ is vertex-transitive and  for a particular value of $t,$ $\theta_{n,m,t}(C_n(R))$ is the resultant of uniform rotation of $m$ copies of circulant subgraphs $\Gamma_j$ of $C_n(R)$ and by Theorem 5.2 in \cite{v2-1}, $\theta_{n,m,t}(C_n(R))$ satisfies the symmetric equidistance condition w.r.t. each of its vertices if and only if $\theta_{n,m,t}(C_n(R))$ = $C_n(S')$ for some $S' \subseteq [1, \frac{n}{2}]$, $r \in R$, $m > 1$ divides $\gcd(n,r)$, $m^3$ divides $n$, $0 \leq j \leq m-1$ and $0 \leq t \leq \frac{n}{m}-1$. Based on the above, we obtain some more properties of $\theta_{n,m,t}(C_n(R))$ as given below. 

\vspace{.1cm} 
\item [\rm (d)] \label{r11d} 	If $t_1$ is the smallest positive integer value of $t$ such that $\theta_{n,m,t_1}(C_n(R))$ = $C_n(S_1)$ for some $S_1 \neq R$, $q_1t_1$ = $\frac{n}{m}$ and $q_1\in\mathbb{N}$, then $\forall i\in\mathbb{N}_0 \ni$ $0 \leq it_1 \leq \frac{n}{m}-1$, $\theta_{n,m,it_1}(C_n(R))$ = $C_n(S_i)$ for some $S_i \subseteq [1, \frac{n}{2}]$, $S_0$ = $R$, $r\in R$, $m$ divides $\gcd(n,r)$, $m^3$ divides $n$, $0 \leq i \leq q_1$ and $1 \leq t_1 \leq \frac{n}{m}$. 

\item [\rm (e)]  \label{r11e}  If $t_2$ is the smallest positive integer value of $t$ such that $\theta_{n,m,t_2}(C_n(R))$ = $C_n(S_1)$ for some $S_1 \neq R$, $C_n(S_1)\in T1_n(C_n(R))$ = $Ad_n(C_n(R))$, $q_2t_2$ = $\frac{n}{m}$ and $q_2\in\mathbb{N}$, then $\forall$ $i\in\mathbb{N}_0$ such that $0 \leq it_2 \leq \frac{n}{m}-1$, $\theta_{n,m,it_2}(C_n(R))$ = $C_n(S_i)$ and $C_n(S_i)\in T1_n(C_n(R))$ for some $S_i \subseteq [1, \frac{n}{2}]$, $S_0$ = $R$, $r\in R$, $m > 1$ divides $\gcd(n,r)$, $m^3$ divides $n$, $0 \leq i \leq q_2$ and $1 \leq t_2 \leq \frac{n}{m}$. 
\end{enumerate}

\begin{definition}\quad \label{d25} Let $T2_{n,m}(C_n(R))$ = $\{C_n(R)\}$ $\cup$ $\{C_n(S):$ $C_n(S)$ is Type-2 isomorphic to $C_n(R)$ w.r.t. $m\}$ where $r\in R$, $m > 1$ divides $\gcd(n,r)$ and $m^3$ divides $n$. We call $T2_{n,m}(C_n(R))$ as {\em the Type-2 set of $C_n(R)$
w.r.t. $m$}.

That is, {\em the Type-2 set of $C_n(R)$ w.r.t. $m$} denoted by $T2_{n,m}(C_n(R))$ is $\{C_n(R)\}$ $\cup$ $\{\theta_{n,m,t}(C_n(R)):$ $\theta_{n,m,t}(C_n(R))$ = $C_n(S)$ and $C_n(S)$ is Type-2 isomorphic to $C_n(R)$ w.r.t. $m$, $1 \leq t \leq \frac{n}{m}-1\}$ = $\{\theta_{n,m,0}(C_n(R))\}$ $\cup$ $\{\theta_{n,m,t}(C_n(R)):$ $\theta_{n,m,t}(C_n(R))$ = $C_n(S)$ and $C_n(S)$ is Type-2 isomorphic to $C_n(R)$ w.r.t. $m$, $1 \leq t \leq \frac{n}{m}-1\}$ where $r\in R$, $m > 1$ divides $\gcd(n,r)$ and $m^3$ divides $n$.
\end{definition}

 We continue to present different properties related to $T2_{n,r}(C_n(R))$.

\begin{enumerate}
\item [\rm (f)]  \label{r11f}  If $t_3$ is the smallest positive integer value of $t$ $\ni$ $\theta_{n,m,t_3}(C_n(R))$ = $C_n(S_1)$ for some $S_1 \neq R$, $C_n(S_1)\in T2_{n,m}(C_n(R))$, $q_3t_3$ = $\frac{n}{m}$ and $q_3\in\mathbb{N}$, then $\forall i\in\mathbb{N}_0 \ni 0 \leq it_3 \leq \frac{n}{m}-1$, $\theta_{n,m,it_3}(C_n(R))$ = $C_n(S_i)$ for some $S_i \subseteq [1, \frac{n}{2}]$, $C_n(S_i)\in T2_{n,m}(C_n(R))$, $S_0$ = $R$, $m > 1$ divides $\gcd(n,r)$, $m^3$ divides $n$, $0 \leq i \leq q_3$ and $1 \leq t_3 \leq \frac{n}{m}-1$. 

\item [\rm (g)]  \label{r11g}  If $t_4$ is the smallest positive integer value of $t$ $\ni$ $\theta_{n,m,t_4}(C_n(R))$ = $C_n(R)$ = $\theta_{n,m,0}(C_n(R))$, $q_4t_4$ = $\frac{n}{m}$, $1 \leq t_4 \leq \frac{n}{m}-1$ and $q_4\in\mathbb{N}$, then $\forall i\in\mathbb{N}_0 \ni 0 \leq it_4 \leq \frac{n}{m}-1$, $\theta_{n,m,it_4}(C_n(R))$ = $C_n(R)$, $r\in R$, $m > 1$ divides $\gcd(n,r)$, $m^3$ divides $n$, $0 \leq i \leq q_4$ and $1 \leq t_4 \leq \frac{n}{m}-1$.  
\end{enumerate}

\vspace{.1cm}
The following result related to $T2_{n,m}(C_n(R))$ is a result similar to Theorem \ref{a7b} related to $T1_n(C_n(R))$.

 \begin{theorem}{\rm \cite{v24}} \quad  \label{t26} {\rm  Let $C_n(R)$ $\cong$ $C_n(S)$, $R \neq S$ and $|R| = |S| \geq 3$. Then, $C_n(S)\in$ $T2_{n,r}($ $C_n(R))$ if and only if  $T2_{n,r}(C_n(R))$ = $T2_{n,r}(C_n(S))$ if and only if  $C_n(R)\in T2_{n,r}(C_n(S))$.  }
 \end{theorem}
\begin{proof} \quad When $T2_{n,m}(C_n(R))$ = $T2_{n,m}(C_n(S))$, $C_n(R)\in T2_{n,m}(C_n(S))$ and $C_n(S)\in T2_{n,m}($ $C_n(R))$. On the other hand, assume that $C_n(S)\in T2_{n,m}(C_n(R))$, $S \neq R$. Our aim is to prove that $T2_{n,m}(C_n(R))$ = $T2_{n,m}(C_n(S))$. That is to prove $T2_{n,m}(C_n(R))$ $\subseteq$ $T2_{n,m}(C_n(S))$ and $T2_{n,m}(C_n(S))$ $\subseteq$ $T2_{n,m}(C_n(R))$. Let $C_n(T)\in T2_{n,m}(C_n(R))$ for some $T$, $T \neq R$. Given, $C_n(S)\in T2_{n,m}(C_n(R))$, $S \neq R$. This implies, $C_n(T)$ = $\theta_{n,m,t}(C_n(R))$, $C_n(S)$ = $\theta_{n,m,t'}(C_n(R))$ (which implies, $C_n(R)$ = $\theta_{n,m,\frac{n}{m}-t'}(C_n(S))$ using property $(c)$,  $0 \leq t' \leq \frac{n}{m}-1$) for some $t$ and $t'$,  $C_n(T)$ and $C_n(R)$ are Type-2 isomorphic w.r.t. $m$ as well as $C_n(S)$ and $C_n(R)$, $S,T \neq R$ and $1 \leq t,t' \leq \frac{n}{m}-1$. This implies that $C_n(T)$ = $\theta_{n,m,t}(C_n(R))$ = $\theta_{n,m,t}(\theta_{n,m,\frac{n}{m}-t'}(C_n(S)))$ = $\theta_{n,m,t+(\frac{n}{m}-t')}(C_n(S))$ = $\theta_{n,m,t_2}(C_n(S))$ for some $t$,$t_1$ = $\frac{n}{m}-t'$ and $t_2$ = $t+(\frac{n}{m}-t')$ and $C_n(S)$, $C_n(T)$ and $C_n(R)$ are isomorphic of Type-2 w.r.t. $m$, $S,T \neq R$ and $1 \leq t,t',t_1,t_2 \leq \frac{n}{m}-1$. This implies, $C_n(T)$ = $\theta_{n,m,t_2}(C_n(S))$ for some $t_2$ and $C_n(T)$ and $C_n(S)$ are Type-2 isomorphic w.r.t. $m$, $1 \leq t_2 \leq \frac{n}{m}-1$. This implies, $C_n(T) \in T2_{n,m}((C_n(S)))$. This implies, $T2_{n,m}(C_n(R))$ $\subseteq$ $T2_{n,m}(C_n(S))$. Similarly, we can prove that $T2_{n,m}(C_n(S))$ $\subseteq$ $T2_{n,m}(C_n(R))$. This implies, $T2_{n,m}(C_n(R))$ = $T2_{n,m}(C_n(S))$.
\end{proof} 

\begin{enumerate}
\item [\rm (h)]  \label{r11i}  $T2_{n,m}(C_n(R))$ is closed under $`\circ'$ follows from property $(f)$ where $\circ$ is defined as $\forall$ $C_n(S)$ = $\theta_{n,m,t_1}(C_n(R)), C_n(T) = \theta_{n,m,t_2}(C_n(R))\in T2_{n,m}(C_n(R))$, $\theta_{n,m,t_1}(C_n(R)) \circ \theta_{n,m,t_2}(C_n(R))$ = $\theta_{n,m,t_3}(C_n(R))$, $t_3$ = $t_1+t_2$, $0 \leq t_1,t_2,t_3 \leq \frac{n}{m}-1$ and $t_3$ in $\theta_{n,m,t_3}$ is calculated under arithmetic modulo $\frac{n}{m}$. 
		
\item [\rm (i)]  \label{r11j} Let $C_n(R)$ $\cong$ $C_n(S)$, $R \neq S$ and $|R| = |S| \geq 3$. Then, either $T2_{n,m}(C_n(R))$ $\cap$ $T2_{n,m}(C_n(S))$ = $\emptyset$ or $T2_{n,m}(C_n(R))$ = $T2_{n,m}(C_n(S))$. This follows from property $(h)$.		
\end{enumerate}

$C_n(R)$ has Type-2 isomorphic circulant graph w.r.t. $m$ if and only if $T2_{n,m}(C_n(R))$ $\neq$ $\{C_n(R)\}$ if and only if $T2_{n,m}(C_n(R$ $)) \setminus \{C_n(R)\} \neq \emptyset$ if and only if $|T2_{n,m}(C_n(R))| > 1$ \cite{v20}. In the next theorem, we  prove that $(T2_{n,m}(C_n(R)), \circ)$  is a subgroup of $(V_{n,m}(C_n(R)), \circ)$. 

\begin{theorem} \label{t27} {\rm $(T2_{n,m}(C_n(R)),  \circ)$ is a subgroup of $(V_{n,m}(C_n(R)), \circ)$ where $r\in R$, $m > 1$ divides $\gcd(n,r)$ and $m^3$ divides $n$.}
\end{theorem}
\begin{proof}\quad $T2_{n,m}(C_n(R))$ is closed under $`\circ'$ by property $(i)$, $\theta_{n,m,0}(C_n(R))$ = $C_n(R)\in T2_{n,m}(C_n(R))$ is the identity element in $T2_{n,m}(C_n(R))$, $T2_{n,m}(C_n(R)) \subseteq V_{n,m}(C_n(R))$ and $(V_{n,m}(C_n(R)), \circ)$ is an Abelian group by Theorem \ref{ab19}. When $T2_{n,m}(C_n(R))$ = $\{ \theta_{n,m,0}(C_n(R)) = C_n(R) \},$ $(T2_{n,m}(C_n(R)), \circ)$ is the trivial subgroup of $(V_{n,m}(C_n(R)), \circ)$. When $T2_{n,m}(C_n(R))$ $\neq$ $\{ C_n(R) = \theta_{n,m,0}(C_n(R))\}$, then by property $(f)$, there exists $t_1$, the smallest positive integer value of $t$ in $\theta_{n,m,t}(C_{n}(R))$ $\ni$ $\theta_{n,m,t_1}(C_{n}(R))$ = $C_{n}(S_1)$ and $C_n(S_{1})\in T2_{n,m}(C_n(R))$ for some $S_1 \neq R$, $0 \leq t,t_1 \leq \frac{n}{m}-1$, $q_1t_1$ = $\frac{n}{m}$, $r\in R,S_1$,  $m > 1$ divides $\gcd(n,r)$ and $m^3$ divides $n$. Let $C_n(S),C_n(T)\in T2_{n,m}(C_n(R))$. Then by property $(f)$, there exists $i$ and $j$ $\ni$ $C_n(S)$ = $\theta_{n,m,it_1}(C_n(R))$ and $C_n(T)$ = $\theta_{n,m,jt_1}(C_n(R))$, $0 \leq t_1 \leq \frac{n}{m}-1$, $q_1t_1$ = $\frac{n}{m}$, $0 \leq i,j \leq q_1-1$ and $0 \leq it_1,jt_1 \leq \frac{n}{m}-1$. This implies, for $0 \leq t_1 \leq \frac{n}{m}-1$, $q_1t_1$ = $\frac{n}{m}$, $0 \leq i,j \leq q_1-1$ and $0 \leq it_1,jt_1 \leq \frac{n}{m}-1$, $\theta_{n,m,it_1}(C_n(R)), \theta_{n,m,jt_1}(C_n(R))\in T2_{n,m}(C_n(R))$. This implies, $\theta_{n,m,it_1}(C_n(R)), \theta_{n,m,(q_1-j)t_1}(C_n(R))\in T2_{n,m}(C_n(R))$ using property $(f)$, $q_1t_1$ = $\frac{n}{m}$, $0 \leq i,j \leq q_1-1$ and $0 \leq it_1,jt_1 \leq \frac{n}{m}-1$. This implies, $\theta_{n,m,it_1}(C_n(R)) \circ \theta_{n,m,(q_1-j)t_1}(C_n(R)) = \theta_{n,m,it_1+(q_1-j)t_1}(C_n(R)) = \theta_{n,m,kt_1}(C_n(R))\in T2_{n,m}(C_n(R))$ by properties $(f)$ and $(i)$ where $\theta_{n,m,(q_1-j)t_1}$ $(C_n(R))$ is the inverse element of $\theta_{n,m,jt_1}(C_n(R))$ = $C_n(T)$ in $T2_{n,m}(C_n(R))$, $q_1t_1$ = $\frac{n}{m}$, $k$ = $q_1-j+i$, $0 \leq i,j,k \leq q_1-1$ and $0 \leq it_1,jt_1,kt_1 \leq \frac{n}{m}-1$. This implies, $(T2_{n,m}(C_n(R)),  \circ)$ is a subgroup of $(V_{n,m}(C_n(R)), \circ)$ where $r\in R$, $m > 1$ divides $\gcd(n,r)$, $m^3$ divides $n$ and $m,n,r\in\mathbb{N}$.
\end{proof} 

\begin{definition} \quad \label{d28} With usual notation, group $(T2_{n,m}(C_n(R)), \circ)$ is called the Type-2 group of $C_n(R)$ w.r.t.  $m$.  
\end{definition}

In \cite{v2-1}, it is proved that circulant graphs $C_{16}(1,2,7)$ and $C_{16}(2,3,5)$ are Type-2 isomorphic w.r.t. $m$ = 2, $Ad_{16}(C_{16}(1,2,7))$ $\neq$ $Ad_{16}(C_{16}(2,3,5)),$ $V_{16,2}(C_{16}(1,2,7))$ = $V_{16,2}(C_{16}(2,3,5))$ and $T2_{16,2}(C_{16}(1,2,7))$ = $\{\theta_{16,2,t}(C_{16}(1,2,7)): t = 0,2\}$ = $\{\theta_{16,2,0}(C_{16}(2,3,5))$ = $C_{16}(2,3,5)$, $\theta_{16,2,2}(C_{16}(2,3,5))$ = $C_{16}(1,2,$ $7)\}$ = $\{\theta_{16,2,t}(C_{16}(2,3,5)): t = 0,2\}$ = $T2_{16,2}(C_{16}(2,3,5)).$ This also implies, $(T2_{16,2}(C_{16}(1,2,7)), \circ)$ = $(T2_{16,2}(C_{16}(2,3,5)), \circ)$ is the Type-2 group on $C_{16}(1,2,7)$ or on $C_{16}(2,3,5)$ w.r.t. $m$ = 2. In the following, we present more results on Type-2 groups on Type-2 isomorphic circulant graphs of $C_n(R)$ w.r.t. $m$ = 2,3,5,7.
	
Corresponding to Theorems \ref{t13}, \ref{t16}, \ref{t18} and \ref{t20}, we obtain Theorems \ref{t29} to \ref{t32} on Type-2 groups, respectively. These results are easy to prove using the definition of $\theta_{n,m,t_1} \circ \theta_{n,m,t_2}$ = $\theta_{n,m,t_1+t_2}$ where $t_1+t_2$ is calculated under $mod~ \frac{n}{m}$. Theorem \ref{t33} is a more general result on the existence of Type-2 isomorphic circulant graphs and its Abelian group.

\begin{theorem}{\rm \quad \label{t29} For $n \geq 2$, $k \geq 3$, $1 \leq 2s-1 \leq 2n-1$, $n \neq 2s-1$, $R$ = $\{2, 2s-1, 4n-(2s-1)\}$ and $S$ = $\{2, 2n-(2s-1), 2n+2s-1\}$, $C_{8n}(R)$ and $C_{8n}(S)$ are Type-2 isomorphic w.r.t. $m$ = 2,  $T2_{8n,2}(C_{8n}(R))$ = $T2_{8n,2}(C_{8n}(S))$ = $\{C_{8n}(R), C_{8n}(S)\}$ and $(T2_{8n,2}(C_{8n}(R)), \circ)$ = $(T2_{8n,2}(C_{8n}(S)), \circ)$ is a Type-2 group, $n,s\in\mathbb{N}$. \hfill $\Box$ }
\end{theorem}

\begin{theorem}  \label{t30} {\rm For $R$ = $\{1, 3, 9n-1, 9n+1\}$, $S$ = $\{3, 3n+1, 6n-1, 12n+1\}$, $T$ = $\{3, 3n-1, 6n+1$, $12n-1\}$ and $n\in\mathbb{N}$, $C_{27n}(R)$, $C_{27n}(S)$ and $C_{27n}(T)$ are Type-2 isomorphic circulant graphs w.r.t. $m$ = 3, $T2_{27n,3}(C_{27n}(R))$ = $T2_{27n,3}(C_{27n}(S))$ = $T2_{27n,3}(C_{27n}(T))$ = $\{C_{8n}(R), C_{8n}(S), C_{8n}(T)\}$ and $(T2_{27n,3}(C_{27n}(R)), \circ)$ = $(T2_{27n,3}(C_{27n}(S)), \circ)$ = $(T2_{27n,3}(C_{27n}(S)), \circ)$ is a Type-2 group. \hfill $\Box$}
\end{theorem}

\begin{theorem}  \label{t31} {\rm For $i$ = 1 to 5, $R_i$ = $\{5, d_i, 25n-d_i, 25n+d_i, 50n-d_i, 50n+d_i\}$, $d_i$ = $5n(i-1)+1$ and $n\in\mathbb{N}$, $C_{125n}(R_{i})$ are Type-2 isomorphic circulant graphs w.r.t. $m$ = 5, $T2_{125n,5}(C_{125n}(R_i))$ = $\{C_{125n}(R_j): j = 1,2,3,4,5\}$ and $(T2_{125n,5}(C_{125n}(R_i)), \circ)$ is a Type-2 group.  		\hfill $\Box$}
\end{theorem}

\begin{theorem}  \label{t32} {\rm For $i$ = 1 to 7, $R_i$ = $\{7, d_i, 49n-d_i, 49n+d_i, 98n-d_i, 98n+d_i, 147n-d_i, 147n+d_i\}$, $d_i$ = $7n(i-1)+1$ and $n\in\mathbb{N}$, $C_{343n}(R_{i})$ are Type-2 isomorphic circulant graphs w.r.t. $m$ = 7, $T2_{343n,7}(C_{343n}(R_i))$ = $\{C_{343n}(R_j): j = 1,2,\dots,7\}$ and $(T2_{343n,7}(C_{343n}(R_i)), \circ)$ is a Type-2 group. \hfill $\Box$}
\end{theorem}

\begin{theorem} \cite{v2-9}  \label{t33} {\rm Let $p$ be an odd prime number, $1 \leq i \leq p$, $1 \leq x \leq p-1$, $y\in\mathbb{N}_0$, $0 \leq y \leq np - 1$, $1 \leq x+yp \leq np^2-1$, $d^{np^3, x+yp}_i$ = $(i-1)xpn+$ $x+yp$ and $R^{np^3, x+yp}_i$ = $\{p$, $d^{np^3,x+yp}_i$, $np^2-d^{np^3,x+yp}_i$, $np^2+d^{np^3, x+yp}_i$, $2np^2-$ $d^{np^3, x+yp}_i$, $2np^2+d^{np^3, x+yp}_i$, $3np^2-d^{np^3, x+yp}_i$, $3np^2+d^{np^3, x+yp}_i$, . . . , $(p-1)np^2$ - $d^{np^3, x+yp}_i$, $(p-1)np^2+d^{np^3, x+yp}_i$, $np^3-d^{np^3, x+yp}_i$, $np^3-p\}$. Then, for $i$ = 1 to $p$, $T2_{np^3, p}(C_{np^3}(R^{np^3, x+yp}_i))$ = $\{\theta_{np^3,p,jn}(C_{np^3}(R^{np^3,x+yp}_i))$ = $C_{np^3}(R^{np^3,x+yp}_{i+j}) :$ $j$ = $0,1,...,p-1$ and $i+j$ in $C_{np^3}(R^{np^3,x+yp}_{i+j})$ is calculated under addition modulo $p \}$ and $(T2_{np^3, p}(C_{np^3}(R^{np^3, x+yp}_i)), \circ)$ is a Type-2 group of order $p$.  \hfill $\Box$   }
\end{theorem}

\medskip
\noindent
{\bf Notation.} \label{5.10} We introduce the following notation related to Type-2 isomorphism of $C_n(R)$ w.r.t. $m$.

$C_n(R) \cong_{T2_{n,m,t}} C_{n}(S)$ when $C_n(R)$ and $C_n(S)$ are Type-2 isomorphic w.r.t. $m$ and $T2_{n,m,t}(C_n(R))$ = $\theta_{n,m,t}(C_{n}(R))$ = $C_n(S)$ for some $t$ where $m > 1$ divides $\gcd(n,r)$, $m^3$ divides $n$, $r\in R,S$ and $1 \leq t \leq \frac{n}{m}-1$.

\section{Examples of Type-1 and Type-2 groups on isomorphic circulant graphs}

 To enlighten the topic, we present here examples of Type-1 and Type-2 sets of circulant graphs and Type-1 and Type-2 groups on isomorphic circulant graphs of Type-1 and Type-2, respectively. In \cite{v2-1} to \cite{v2-5}, we obtained many  isomorphic circulant graphs of Type-1 and Type-2 and we present  Type-1 and Type-2 sets and groups in problem \ref{p3} corresponding to isomorphic circulant graphs that occurs in problems in \cite{v2-1}. Similarly, one can present more Type-1 and Type-2 sets and groups corresponding to isomorphic circulant graphs that occurs in problems in \cite{v2-2} to \cite{v2-5}. 

\begin{prm}  \label{p1} {\rm For the circulant graphs $C_{54}(R_1)$, $C_{54}(R_2)$ and $C_{54}(R_3)$ with $R_1$ = $\{2,3,16,20\}$, $R_2$ = $\{3,4,14,22\}$ and $R_3$ = $\{3,8,10,26\}$, show that Type-1 and Type-2 sets of circulant graphs as given by
\begin{enumerate}
	\item [\rm (i)] $T2_{54,3}(C_{54}(R_1))$ = $\{C_{54}(R_1)$, $C_{54}(R_2), C_{54}(R_3)\}$ = $T2_{54,3}(C_{54}(R_2))$ = $T2_{54,3}(C_{54}(R_3))$;  
	
	\item [\rm (ii)]  $T1_{54}(C_{54}(R_1))$ = $\{C_{54}(R_1), C_{54}(A_1), C_{54}(B_1)\}$ = $\{C_{54}(xR_1): x = 1, 5, 7\}$,
	\\
	$T1_{54}(C_{54}(R_2))$ = $\{C_{54}(R_2), C_{54}(A_2), C_{54}(B_2)\}$ = $\{C_{54}(xR_2): x = 1, 5, 7\}$ and 
	\\
	$T1_{54}(C_{54}(R_3))$ = $\{C_{54}(R_3), C_{54}(A_3), C_{54}(B_3)\}$ = $\{C_{54}(xR_3): x = 1, 5, 7\}$  where 
	
	$C_{54}(A_1)$ = $C_{54}(8,10,15,26)$, $C_{54}(A_2)$ = $C_{54}(2,15,16,20)$, $C_{54}(A_3)$ = $C_{54}(4,14,15,22)$, 

	$C_{54}(B_1)$ = $C_{54}(4,14,21,22)$, $C_{54}(B_2)$ = $C_{54}(8,10,21,26)$, $C_{54}(B_3)$ = $C_{54}(2,16,20,21)$; and 
	
	\item [\rm (iii)] For $i$ = 1,2,3, $(T1_{54}(C_{54}(R_i)), \circ')$ and $(T2_{54,3}(C_{54}(R_i)), \circ)$ are Abelian groups.	 
\end{enumerate}	  }   
\end{prm}

\noindent
{\bf Solution.} \quad $(i)$ Here, $n$ = 54, $r = 3$ and $m$ = $\gcd(n, r)$ = $\gcd(54, 3)$ = 3 = $r$. Let

$S_1$ = $R_1 \cup (54-R_1)$ = $\{2,3,16,20, 34,38,51,52\}$, 

$S_2$ = $R_2 \cup (54-R_2)$ = $\{3,4,14,22,$ $32,40,50,51\}$ and 

$S_3$ = $R_3 \cup (54-R_3)$ = $\{3,8,10,26, 28,44,46,51\}$. For $i$ = 1 to 3, we have $C_{54}(S_i)$ = $C_{54}(R_i)$ and 

\noindent
 $V_{54,3}(C_{54}(R_1))$ = $V_{54,3}(C_{54}(2,3,16,20))$ 

\hfill = $\{\theta_{54,3,t}(C_{54}(2,3,16,20)): t = 0,1,\dots,\frac{54}{\gcd(54,3)}-1 = 17\}$ = $V_{54,3}(C_{54}(R_2))$ = $V_{54,3}(C_{54}(R_2))$.

 For various values of $t$, it is easy to see from Table 1 that $\theta_{54,3,i+6t}(C_{54}(S_1))$ = $\theta_{54,3,i}(C_{54}(S_1))$ and $\theta_{54,3,2(j-1)}(C_{54}(S_1))$ = $C_{54}(S_j)\in V_{54,3}(C_{54}(S_1))$, $0 \leq i \leq 5$, $0 \leq i+6t \leq$ $\frac{54}{\gcd(54,3)}-1 = 17$ and $1 \leq j \leq 3$. This implies, $V_{54,3}(C_{54}(R_1))$ = $\{\theta_{54,3,t}(C_{54}(2,3,16,$ $20)): t = 0,1,\dots,5 \}$ = $V_{54,3}(C_{54}(R_2))$ = $V_{54,3}(C_{54}(R_3))$, and $C_{54}(R_1)$, $\theta_{54,3,2}(C_{54}(R_1))$ = $C_{54}(R_2)$ and $\theta_{54,3,4}(C_{54}(R_1))$ = $C_{54}(R_3)$ are isomorphic circulant graphs. 
\\
 $T1_{54}(C_{54}(R_1))$ =  $T1_{54}(C_{54}(2,3,16,20))$ = $\{C_{54}(2,3,16,20), C_{54}(8,10,15,26), C_{54}(4,14,21,22)\}$. 

$\Rightarrow$ $C_{54}(R_2),C_{54}(R_3)\notin T1_{54}(C_{54}(R_1))$.

$\Rightarrow$ $C_{54}(R_i)$ are Type-2 isomorphic w.r.t. $m$ = 3 for $i$ = 1 to 3. Clearly, 
\\
$T2_{54,3}(C_{54}(R_i))$ = $\{C_{54}(R_1),$ $C_{54}(R_2)$, $C_{54}(R_3)\}$ $\subset$ $V_{54,3}(C_{54}(R_i))$ = $V_{54,3}(C_{54}(R_j))$ and 
\\
$(T2_{54,3}(R_i), \circ)$ is a subgroup of Abelian group $(V_{54,3}(C_{54}(R_j)), \circ)$, $1 \leq i,j \leq 3$. 

Type-2 isomorphic circulant graphs $C_{54}(2,3,16,20)$, $C_{54}(3,4,14,22)$, $C_{54}(3,8,10,26)$  w.r.t. $m$ = 3 are given in Figures 1, 2, 3.

\begin{figure}[ht]
	\centerline{\includegraphics[width=6.2in]{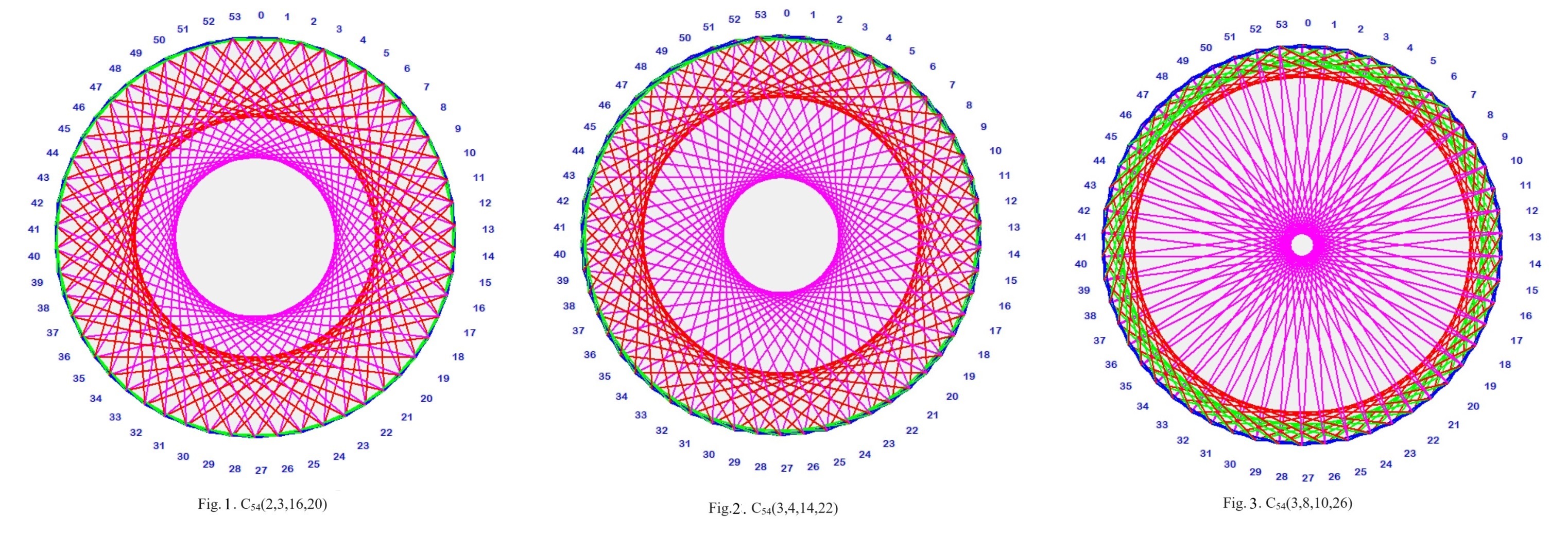}}
\end{figure}
\begin{table} \label{t1} 
\caption{\small{Calculation of $\theta_{54,3,t}(R_1 \cup (54-R_1))$, $R_1$ = $\{2,3,16,20\}$, $t$ = 0 to 6.}}
\begin{center}
\scalebox{0.7}{
\begin{tabular}{||c||c||c|c|c|c|c|c|c|c||c||} \hline \hline
~ \hspace{.05cm} $t$ \hspace{.1cm} & \backslashbox{$\theta_{54,3,t}(x)$}{Jump  size \\ $x$}
& \hspace{.1cm} 2 \hspace{.1cm} & \hspace{.1cm} 3 \hspace{.1cm} & \hspace{.1cm} 16 \hspace{.1cm} & \hspace{.1cm} 20 \hspace{.1cm} & \hspace{.1cm} 34 \hspace{.1cm} & \hspace{.1cm} 38 \hspace{.1cm} & \hspace{.1cm} 51 \hspace{.1cm} & \hspace{.1cm} 52 \hspace{.1cm} & Symmetric or NS \\
\cline{1-9}
  \hline \hline
& & &  &   &  &  & & & &  \\
t & $\theta_{54,3,t}(x)$ & 2+6t & 3 & 16+3t & 20+6t & 34+3t & 38+6t & 51 & 52+3t & {$T1$ or  $T2$ or  NS} \\ \hline \hline
& & &  &   &  &  & & & &  \\
0 & $\theta_{54,3,0}(x)$ & 2 & 3 & 16 & 20 & 34 & 38 & 51 & 52 & Yes (Identity)  \\\hline
 & & &  &   &  &  & & & & \\
1 & $\theta_{54,3,1}(x)$ & 8 & 3 & 19 & 26 & 37 & 44 & 51 & 1 & NS \\\hline 
& & &  &   &  &  & & & & \\
2 & $\theta_{54,3,2}(x)$ & 14 & 3 & 22 & 32 & 40 & 50 & 51 & 4  & Yes (Type-2) \\\hline
& & &  &   &  &  & & &  & \\
3 & $\theta_{54,3,3}(x)$ & 20 & 3 & 25 & 38 & 43 & 2 & 51 & 7 & NS \\\hline
& & &  &   &  &  & & & & \\
4 & $\theta_{54,3,4}(x)$ & 26 & 3 & 28 & 44 & 46 & 8 & 51 & 10 & Yes (Type-2) \\\hline
& & &  &   &  &  & & & & \\
5 & $\theta_{54,3,5}(x)$ & 32 & 3 & 31 & 50 & 49 & 14 & 51 & 13 & NS \\\hline
& & &  &   &  &  & & & & \\
6 & $\theta_{54,3,6}(x)$ & 38 & 3 & 34 & 2 & 52 & 20 & 51 & 16 & Yes (Identity) \\\hline\hline
\end{tabular}}
\end{center}
\footnotesize{T1: Type-1; T2: Type-2 isomorphic w.r.t. $r$ = 3; NS: Non-symmetric.}
\end{table} 

\vspace{.2cm}
\noindent
$(ii)$ We have

\hspace{.25cm} $T1_{54}(C_{54}(R_1))$ = $T1_{54}(C_{54}(2,3,16,20))$ = $\{C_{54}(x(2,3,16,20)): x\in\varphi_{54}\}$

\hspace{2.6cm} = $\{C_{54}(x(2,3,16,20)): x = 1, 5, 7, 11, 13, 17, 19, 23\}$ 

\hspace{2.6cm} = $\{C_{54}(2,3,16,20), C_{54}(8,10,15,26), C_{54}(4,14,21,22)\}$

\hspace{2.6cm} = $\{C_{54}(R_1), C_{54}(A_1), C_{54}(B_1)\}$ = $\{C_{54}(xR_1): x = 1, 5, 7\}$ 

\hspace{2.6cm} = $T1_{54}(C_{54}(A_1))$ = $T1_{54}(C_{54}(B_1))$ 
\\
where $A_1$ = $\{8,10,15,26\}$ and $B_1$ = $\{4,14,21,22\}$; 

\hspace{.25cm} $T1_{54}(C_{54}(R_2))$ = $T1_{54}(C_{54}(3,4,14,22))$ = $\{C_{54}(x(3,4,14,22)): x\in\varphi_{54}\}$

\hspace{2.6cm} = $\{C_{54}(x(3,4,14,22)): x = 1, 5, 7, 11, 13, 17, 19, 23\}$ 

\hspace{2.6cm} = $\{C_{54}(3,4,14,22), C_{54}(2,15,16,20), C_{54}(8,10,21,26)\}$

\hspace{2.6cm} = $\{C_{54}(R_2), C_{54}(A_2), C_{54}(B_2)\}$ = $\{C_{54}(xR_2): x = 1, 5, 7\}$ 

\hspace{2.6cm} = $T1_{54}(C_{54}(A_2))$ = $T1_{54}(C_{54}(B_2))$ 
\\
where $A_2$ = $\{2,15,16,20\}$ and $B_2$ = $\{8,10,21,26\}$; 

\hspace{.25cm} $T1_{54}(C_{54}(R_3))$ = $T1_{54}(C_{54}(3,8,10,26))$ = $\{C_{54}(x(3,8,10,26)): x\in\varphi_{54}\}$

\hspace{2.6cm} = $\{C_{54}(x(3,8,10,26)): x = 1, 5, 7, 11, 13, 17, 19, 23\}$ 

\hspace{2.6cm} = $\{C_{54}(3,8,10,26), C_{54}(4,14,15,22), C_{54}(2,16,20,21)\}$

\hspace{2.6cm} = $\{C_{54}(R_3), C_{54}(A_3), C_{54}(B_3)\}$ = $\{C_{54}(xR_3): x = 1, 5, 7\}$ 

\hspace{2.6cm} = $T1_{54}(C_{54}(A_3))$ = $T1_{54}(C_{54}(B_3))$ 
\\
where $A_3$ = $\{4,14,15,22\}$ and $B_3$ = $\{2,16,20,21\}$. 

\vspace{.1cm}
\noindent
$(iii)$ Result follows from definition \ref{a6} and Theorem \ref{t5.6}. \hfill $\Box$

\begin{prm} \quad \label{p2} {\rm For $R_1$ = $\{3,7,20,34\}$, 
\\
(a) find the following sets: (i) $T1_{81}(C_{81}(R_1))$,  (ii) $V_{81,3}(C_{81}(R_1))$, (iii) $T2_{81,3}(C_{81}(R_1))$; and
\\
(b) prove that (i) $(T1_{81}(C_{81}(R_1)), \circ')$, (ii) $(T2_{81,3}(C_{81}(R_1)), \circ)$, and (iii) $(V_{81,3}(C_{81}(R_1)), \circ)$  are Abelian 

 groups.   } 
\end{prm}

\noindent
{\bf Solution.}\quad Let $R_1$ = $\{3,7,20,34\}$ and $S_1$ = $R_1 \cup (81 - R_1)$ = $\{3,7,20,34, 47,61,74,78\}$. 

Here, $n$ = 81 = $3\times 3^3$ and so possible value of $m$ is $m$ = 3 = $\gcd(54, 3)$ = $\gcd(n, r)$ = $r \in R_1$.  

\vspace{.1cm}
\noindent
{\bf Case $a (i)$ $\&$ Case $b (ii)$:} Consider,

$T1_{81}(C_{81}(R_1))$ = $T1_{81}(C_{81}(3,7,20,34))$ = $\{C_{81}(x(3,7,20,34)): x\in\varphi_{81}\}$ 

\hspace{2.25cm} = $\{C_{81}(x(3,7,20,34)): x = 1,2, 4,5, 7,8, 10,11, 13,14, 16,17, 19,20,$ 

\hfill $22,23, 25,26, 28,29, 31,32, 34,35, 37,38, 39,40\}$

\hfill = $\{C_{81}(3,7,20,34), C_{81}(6, 13, 14, 40), C_{81}(1, 12, 26, 28), C_{81}(8, 15, 19, 35), C_{81}(5, 21, 22, 32),$

\hfill $C_{81}(2, 24, 25, 29), C_{81}(11, 16, 30, 38), C_{81}(4, 23, 31, 33), C_{81}(10, 17, 37, 39)\}$ 

\hspace{2.25cm}  = $\{C_{81}(xR_1): x = 1,2, 4,5, 7,8, 10,11, 13\}$.

Clearly, $(T1_{81}(C_{81}(R_1)), \circ)$ =  $(\{C_{81}(x(3, 7, 20, 34)): x = 1,2, 4,5, 7,8, 10,11, 13\}, \circ')$ is the Type-1 group of $C_{81}(R_1)$ which is an Abelian group.

\vspace{.1cm}
\noindent
{\bf Case $a (ii)$ $\&$ Case $b (iii)$:}  We have $V_{81,3}(C_{81}(R_1))$ = $\{\theta_{81,3,t}(C_{81}(R_1)):$ $t$ = 0 to $\frac{n}{\gcd(n,r)}$-1 = 26$\}$ and from Table 2,   

$\theta_{81,3,0}(S_1)$ = $S_1$, $\theta_{81,3,3}(S_1)$ = $\{3,11,16,38, 43,65,70,78\}$ = $S_2$,   

$\theta_{81,3,6}(S_1)$ = $\{2,3,25,29, 52,56,78,79\}$ = $S_3$, say,  

$\theta_{81,3,i+9t}(S_1)$ = $\theta_{81,3,i}(S_1)$, $\theta_{81,3,3(j-1)}(S_1)$ = $S_j$ and 

$C_{81}(S_j)\in V_{81,3}(C_{81}(S_1))$, $1 \leq j \leq 3$ and $0 \leq i \leq 8$. See Table 2.

Clearly, $(V_{81,3}(C_{81}(R_1)), \circ)$ =  $\{\theta_{81,3,t}(C_{81}(3, 7, 20, 34)):$ $t$ = 0 to 26$\}$ is an Abelian group.

\vspace{.1cm}
\noindent
{\bf Case $a (iii)$ $\&$ Case $b (ii)$:}    Let $R_2$ = $\{3,11,16,38\}$, $R_3$ = $\{2,3,25,29\}$, $S_2$ = $R_2 \cup (81 - R_2)$ and $S_3$ = $R_3 \cup (81 - R_3)$. From Table 2, we get, $\theta_{81,3,i+9t}(S_1)$ = $\theta_{81,3,i}(S_1)$, $\theta_{81,3,3(j-1)}(S_1)$ = $S_j$, $V_{81,3}(S_k)$ = $\{\theta_{81,3,t}(S_k):$ $t$ = 0 to 8$\}$ and $S_j\in V_{81,3}(S_k)$, $1 \leq j,k \leq 3$ and $0 \leq i \leq 8$. 

This implies, $\theta_{81,3,i+9t}(C_{81}(R_1))$ = $\theta_{81,3,i}(C_{81}(R_1))$, $\theta_{81,3,3(j-1)}(C_{81}(R_1))$ = $C_{81}(R_j)$,  $V_{81,3}(C_{81}(R_k))$ = $\{\theta_{81,3,t}(C_{81}(R_k)):$ $t$ = 0 to 8$\}$ and $C_{81}(R_j)\in V_{81,3}(C_{81}(R_k))$, $1 \leq j,k \leq 3$ and $0 \leq i \leq 8$. 

$\Rightarrow$ $C_{81}(R_i)$ are isomorphic circulant graphs for $i$ = 1 to 3. Moreover, 
\\
$T1_{81}(C_{81}(R_1))$ = $T1_{81}(C_{81}(3,7,20,34))$ = $\{C_{81}(x(3,7,20,34)): x\in\varphi_{81}\}$

\hspace{1.85cm} = $\{C_{81}(x(3,7,20,34)): x = 1, 5, 7, 11, 13, 17, 19, 23\}$ 

\hspace{1.85cm} = $\{C_{81}(3,7,20,34), C_{81}(6,13,14,40), C_{81}(1,12,26,28), C_{81}(8,15,19,35)$

\hfill $C_{81}(5,21,22,32), C_{81}(2,24,25,29), C_{81}(11,16,30,38), C_{81}(4,23,31,33)$, 

\hfill $C_{81}(10,17,37,39)\}$ = $T2_{81,3}(R_2)$ = $T2_{81,3}(R_3)$.  

$\Rightarrow$ $C_{81}(3,11,16,38),C_{81}(2,3,25,29)\notin T1_{81}(C_{81}(3,7,20,34))$.

$\Rightarrow$ $C_{81}(3,7,20,34)$, $C_{81}(3,11,16,38)$, $C_{81}(2,3,25,29)$ are Type-2 isomorphic w.r.t. $m$ = 3.

Clearly, $T2_{81,3}(R_i)$ = $T2_{81,3}(R_j)$ $\subset$ $V_{81,3}(C_{81}(R_j))$ = $V_{81,3}(C_{81}(R_i))$ and $(T2_{81,3}(C_{81}(R_i)), \circ)$ is a subgroup of Abelian group $(V_{81,3}(C_{81}(R_i)), \circ)$, $1 \leq i,j \leq 3$. 

\begin{table} \label{t2} 
\caption{\small{Calculation of $\theta_{81,3,t}(R \cup (81-R))$, $R$ = $\{3,7,20,34\}$, $r$ = 0 to 8.}}
\begin{center}
\scalebox{0.7}{
\begin{tabular}{||c||c||c|c|c|c|c|c|c|c|c||} \hline \hline
~ \hspace{.05cm} $t$ \hspace{.1cm} & \backslashbox{$\theta_{81,3,t}(x)$}{Jump size \\ $x$}
& \hspace{.1cm} 3 \hspace{.1cm} & \hspace{.1cm} 7 \hspace{.1cm} & \hspace{.1cm} 20 \hspace{.1cm} & \hspace{.1cm} 34 \hspace{.1cm} & \hspace{.1cm} 47 \hspace{.1cm} & \hspace{.1cm} 61 \hspace{.1cm} & \hspace{.1cm} 74 \hspace{.1cm} & \hspace{.1cm} 78 \hspace{.1cm} & Symmetric or NS \\\hline \hline
& & &  &   &  &  & & & &  \\
t & $\theta_{81,3,t}(x)$ & 3 & 7+3t & 20+6t & 34+3t & 47+6t & 61+3t & 74+6t & 78 & {$T1$ or  $T2$ or  NS} \\ \hline \hline
& & &  &   &  &  & & & &  \\
0 & $\theta_{81,3,0}(x)$ & 3 & 7 & 20 & 34 & 47 & 61 & 74 & 78 & Identity  \\\hline
 & & &  &   &  &  & & & & \\
1 & $\theta_{81,3,1}(x)$ & 3 & 10 & 26 & 37 & 53 & 64 & 80 & 78 & NS \\\hline 
& & &  &   &  &  & & & & \\
2 & $\theta_{81,3,2}(x)$ & 3 & 13 & 32 & 40 & 59 & 67 & 5 & 78 & NS \\\hline
& & &  &   &  &  & & &  & \\
3 & $\theta_{81,3,3}(x)$ & 3 & 16 & 38 & 43 & 65 & 70 & 11 & 78 & Yes (Type-2) \\\hline
& & &  &   &  &  & & & & \\
4 & $\theta_{81,3,4}(x)$ & 3 & 19 & 44 & 46 & 71 & 73 & 17 & 78 & NS \\\hline
& & &  &   &  &  & & & & \\
5 & $\theta_{81,3,5}(x)$ & 3 & 22 & 50 & 49 & 77 & 76 & 23 & 78 & NS \\\hline
& & &  &   &  &  & & & & \\
6 & $\theta_{81,3,6}(x)$ & 3 & 25 & 56 & 52 & 2 & 79 & 29 & 78 & Yes (Type-2) \\\hline
& & &  &   &  &  & & & & \\
7 & $\theta_{81,3,7}(x)$ & 3 & 28 & 62 & 55 & 8 & 1 & 35 & 78 & NS \\\hline
& & &  &   &  &  & & & & \\
8 & $\theta_{81,3,8}(x)$ & 3 & 31  & 68 & 58 & 14 & 4 & 41 & 78 & NS  \\\hline\hline
\end{tabular}}
\end{center}
\footnotesize{T1: Type-1; T2: Type-2 isomorphic w.r.t. $r$ = 3; NS: Non-symmetric.}
\end{table} 
\begin{figure}[ht]
	\centerline{\includegraphics[width=6.2in]{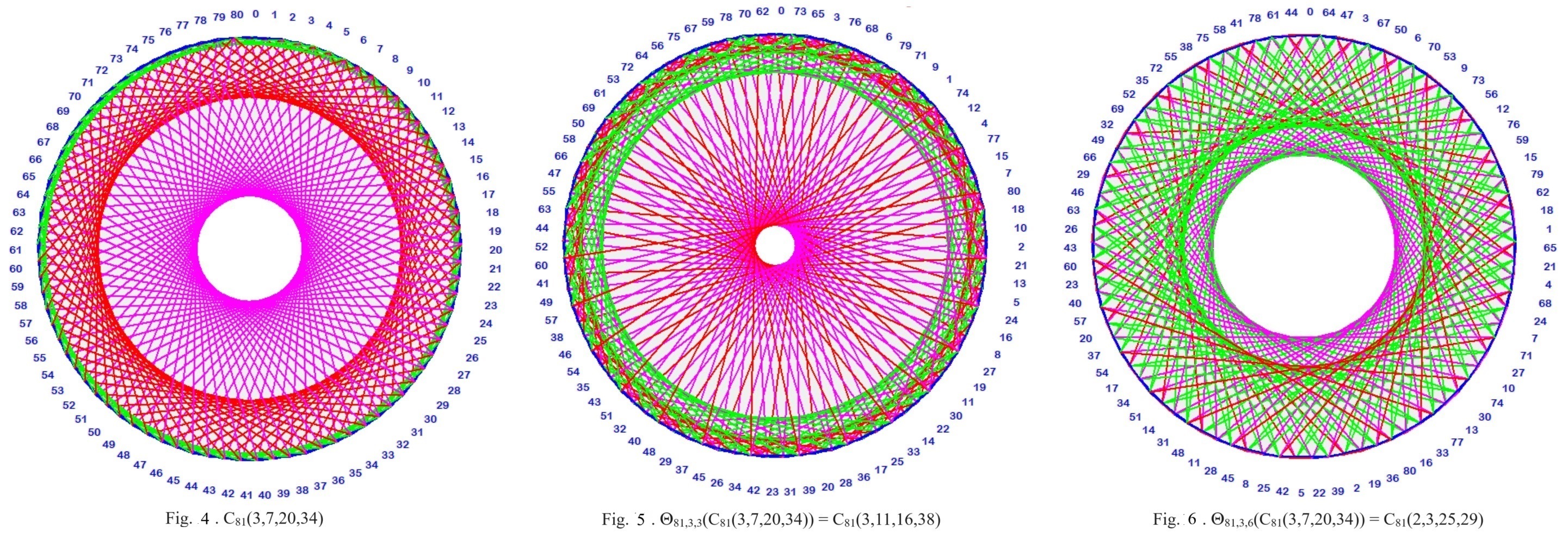}}
\end{figure}

Circulant graphs $C_{81}(3,7,20,34)$, $\theta_{81,3,3}(C_{81}(3,7,20,34))$ = $C_{81}(3,11,16,38)$, $\theta_{81,3,6}(C_{81}(3,7,20,34))$ = $C_{81}(2,3,25,29)$ which are Type-2 isomorphic w.r.t. $m$ = 3 are given in Figures 4, 5, 6.   \hfill $\Box$

\begin{prm} \quad \label{p3} {\rm For the following circulant graphs, find (i) Type-1 set,  (ii) Type-2 set, (iii) Type-1 group, and (iv) Type-2 group.

(a) $C_{16}(1,2,7)$,  (b) $C_{16}(1,2,4,6,7)$,  (c) $C_{16}(1,2,4,7,8)$,

(d) $C_{24}(1,2,8,11)$,  (e) $C_{24}(1,2,10,11)$,   (f) $C_{27}(1,3,8,10)$,

(g) $C_{48}(1,2,23)$,  (h) $C_{48}(1,4,23)$,  (i) $C_{48}(1,6,23)$,

(j) $C_{54}(1,3,17,19)$,  (k) $C_{54}(1,6,17,19)$,  (l) $C_{54}(1,17,18,19)$, 

(m) $C_{108}(3,5,31,41)$,  (n) $C_{108}(5,12,31,41)$, and  (o) $C_{108}(5,18,31,41)$. 
}   
\end{prm}

\noindent
{\bf Solution.} \quad Type-1/Type-2 isomorphic circulant graphs of these graphs are obtained in \cite{v2-1} and using these isomorphic graphs, here we find their corresponding (i) Type-1 set,  (ii) Type-2 set, (iii) Type-1 group, and (iv) Type-2 group as follows. 
\begin{enumerate}
\item [\rm (a)] In problem 1 in \cite{v2-1}, we have
\\
$T1_{16}(C_{16}(1,2,7))$ = $Ad_{16}(C_{16}(1,2,7))$ = $\{C_{16}(1,2,7), C_{16}(3,5,6)\}$ = $T1_{16}(C_{16}(3,5,6))$ and
\\
$T2_{16,2}(C_{16}(1,2,7))$ = $\{C_{16}(1,2,7), C_{16}(2,3,5)\}$ = $T2_{16,2}(C_{16}(2,3,5))$. 
\\
$\Rightarrow$ Type-1 set of $C_{16}(1,2,7)$ is 

$T1_{16}(C_{16}(1,2,7))$ = $\{C_{16}(1,2,7), C_{16}(3,5,6)\}$ = $T1_{16}(C_{16}(3,5,6))$; 
\\
Type-1 group of $C_{16}(1,2,7)$ is 

$(T1_{16}(C_{16}(1,2,7)), \circ')$ = $(T1_{16}(C_{16}(3,5,6)), \circ')$ 
\\ 
where $T1_{16}(C_{16}(1,2,7))$ = $\{C_{16}(1,2,7), C_{16}(3,5,6)\}$ = $T1_{16}(C_{16}(3,5,6))$; and
 \\
Type-2 group of $C_{16}(1,2,7)$ w.r.t. $m$ = 2 is 

$(T2_{16,2}(C_{16}(1,2,7)), \circ)$ = $(T2_{16,2}(C_{16}(2,3,5)), \circ)$ 
\\ 
where $T2_{16,2}(C_{16}(1,2,7))$ = $\{C_{16}(1,2,7), C_{16}(2,3,5)\}$ = $T2_{16,2}(C_{16}(2,3,5))$. 

\item [\rm (b)] In problem 2 in \cite{v2-1}, we have
\\
$T1_{16}(C_{16}(1,2,4,6,7))$ = $\{C_{16}(1,2,4,6,7), C_{16}(2,3,4,5,6)\}$ = $T1_{16}(C_{16}(2,3,4,5,6))$;
\\
$T2_{16,2}(C_{16}(1,2,4,6,7))$ = $\{C_{16}(1,2,4,6,7)\}$ and $T2_{16,2}(C_{16}(2,3,4,5,6))$ = $\{C_{16}(2,3,4,5,6)\}$. 
\\
$\Rightarrow$ Type-1 set of $C_{16}(1,2,4,6,7)$ is 

$T1_{16}(C_{16}(1,2,4,6,7))$ = $\{C_{16}(1,2,4,6,7), C_{16}(2,3,4,5,6)\}$ = $T1_{16}(C_{16}(2,3,4,5,6))$; 
\\
Type-1 group of $C_{16}(1,2,4,6,7)$ is

 $(T1_{16}(C_{16}(1,2,4,6,7)), \circ')$ = $(T1_{16}(C_{16}(2,3,4,5,6)), \circ')$ 
\\
 where $T1_{16}(C_{16}(1,2,4,6,7))$ = $\{C_{16}(1,2,4,6,7), C_{16}(2,3,4,5,6)\}$ = $T1_{16}(C_{16}(2,3,4,5,6))$; 
 \\
Type-2 group of $C_{16}(1,2,4,6,7)$ w.r.t. $m$ = 2 is $(T2_{16,2}(C_{16}(1,2,4,6,7)), \circ)$  and 
\\
Type-2 group of $C_{16}(2,3,4,5,6)$ w.r.t. $m$ = 2 is $(T2_{16,2}(C_{16}(2,3,4,5,6)), \circ)$ where  
\\
$T2_{16,2}(C_{16}(1,2,4,6,7))$ = $\{C_{16}(1,2,4,6,7)\}$ and $T2_{16,2}(C_{16}(2,3,4,5,6))$ = $\{ C_{16}(2,3,4,5,6)\}$. 

\item [\rm (c)] In problem 2 in \cite{v2-1}, we have

$T1_{16}(C_{16}(1,2,4,7,8))$ = $\{C_{16}(1,2,4,7,8), C_{16}(3,4,5,6,8)\}$ = $T1_{16}(C_{16}(3,4,5,6,8))$ and

$T2_{16,2}(C_{16}(1,2,4,7,8))$ = $\{C_{16}(1,2,4,7,8), C_{16}(2,3,4,5,8)\}$ = $T2_{16,2}(C_{16}(2,3,4,5,8))$. 
\\
$\Rightarrow$ Type-1 set of $C_{16}(1,2,4,7,8)$ is 

$T1_{16}(C_{16}(1,2,4,7,8))$ = $\{C_{16}(1,2,4,7,8), C_{16}(3,4,5,6,8)\}$ = $T1_{16}(C_{16}(3,4,5,6,8))$; 
\\
Type-1 group of $C_{16}(1,2,4,7,8)$ is 

$(T1_{16}(C_{16}(1,2,4,7,8)), \circ')$ = $(T1_{16}(C_{16}(3,4,5,6,8)), \circ')$ 
\\
 where $T1_{16}(C_{16}(1,2,4,7,8))$ = $\{C_{16}(1,2,4,7,8), C_{16}(3,4,5,6,8)\}$ = $T1_{16}(C_{16}(3,4,5,6,8))$; 
 \\
Type-2 group of $C_{16}(1,2,4,7,8)$ w.r.t. $m$ = 2 is 

$(T2_{16,2}(C_{16}(1,2,4,7,8)), \circ)$ = $(T2_{16,2}(C_{16}(2,3,4,5,8)), \circ)$ where

\hfill  $T2_{16,2}(C_{16}(1,2,4,7,8))$ = $\{C_{16}(1,2,4,7,8), C_{16}(2,3,4,5,8)\}$ = $T2_{16,2}(C_{16}(2,3,4,5,8))$. 

\item [\rm (d)] In problem 3 in \cite{v2-1}, we have

$T1_{24}(C_{24}(1,2,8,11))$ = $\{C_{24}(1,2,8,11), C_{24}(5,7,8,10)\}$ = $T1_{24}(C_{24}(5,7,8,10))$ and

$T2_{24,2}(C_{24}(1,2,8,11))$ = $\{C_{24}(1,2,8,11), C_{24}(2,5,7,8)\}$ = $T2_{24,2}(C_{24}(2,5,7,8))$. 
\\
$\Rightarrow$ Type-1 set of $C_{24}(1,2,8,11)$ is 

$T1_{24}(C_{24}(1,2,8,11))$ = $\{C_{24}(1,2,8,11), C_{24}(5,7,8,10)\}$ = $T1_{24}(C_{24}(5,7,8,10))$; 
\\
Type-1 group of $C_{24}(1,2,8,11)$ is 

$(T1_{24}(C_{24}(1,2,8,11)), \circ')$ = $(T1_{24}(C_{24}(5,7,8,10)), \circ')$ 
\\ 
where $T1_{24}(C_{24}(1,2,8,11))$ = $\{C_{24}(1,2,8,11), C_{24}(5,7,8,10)\}$ = $T1_{24}(C_{24}(5,7,8,10))$; and
 \\
Type-2 group of $C_{24}(1,2,8,11)$ w.r.t. $m$ = 2 is 

$(T2_{24,2}(C_{24}(1,2,8,11)), \circ)$ = $(T2_{24,2}(C_{24}(2,5,7,8)), \circ)$ 
\\
 where $T2_{24,2}(C_{24}(1,2,8,11))$ = $\{C_{24}(1,2,8,11), C_{24}(2,5,7,8)\}$ = $T2_{24,2}(C_{24}(2,5,7,8))$. 

\item [\rm (e)] In problem 3 in \cite{v2-1}, we have

$T1_{24}(C_{24}(1,2,10,11))$ = $\{C_{24}(1,2,10,11), C_{24}(2,5,7,10)\}$ = $T1_{24}(C_{24}(2,5,7,10))$ and

$T2_{24,2}(C_{24}(1,2,10,11))$ = $\{C_{24}(1,2,10,11)\}$ and $T2_{24,2}(C_{24}(2,5,7,10))$ = $\{C_{24}(2,5,7,10)\}$. 
\\
$\Rightarrow$ Type-1 set of $C_{24}(1,2,10,11)$ is 

$T1_{24}(C_{24}(1,2,10,11))$ = $\{C_{24}(1,2,10,11), C_{24}(2,5,7,10)\}$ = $T1_{24}(C_{24}(2,5,7,10))$; 
\\
Type-1 group of $C_{24}(1,2,10,11)$ is 

$(T1_{24}(C_{24}(1,2,10,11)), \circ')$ = $(T1_{24}(C_{24}(2,5,7,10)), \circ')$ 
\\ 
where $T1_{24}(C_{24}(1,2,10,11))$ = $\{C_{24}(1,2,10,11), C_{24}(2,5,7,10)\}$ = $T1_{24}(C_{24}(2,5,7,10))$; and
 
Type-2 group of $C_{24}(1,2,10,11)$ w.r.t. $m$ = 2 is $(T2_{24,2}(C_{24}(1,2,10,11)), \circ)$ and 

Type-2 group of $C_{24}(2,5,7,10)$ w.r.t. $m$ = 2 is $(T2_{24,2}(C_{24}(2,5,7,10)), \circ)$
\\
where $T2_{24,2}(C_{24}(1,2,10,11))$ = $\{C_{24}(1,2,10,11)\}$ and $T2_{24,2}(C_{24}(2,5,7,10))$ = $\{C_{24}(2,5,7,10)\}$. 

\item [\rm (f)] In problem 4 in \cite{v2-1}, we have
\\
$T1_{27}(C_{27}(1,3,8,10))$ = $\{C_{27}(1,3,8,10), C_{27}(2,6,7,11), C_{27}(4,5,12,13)\}$ 

\hfill = $T1_{27}(C_{27}(2,6,7,11))$ =  $T1_{27}(C_{27}(4,5,12,13))$ and
\\
$T2_{27,3}(C_{27}(1,3,8,10))$ = $\{C_{27}(1,3,8,10), C_{27}(3,4,5,13), C_{27}(2,3,7,11)\}$ 

\hfill  = $T2_{27,3}(C_{27}(3,4,5,13))$ = $T2_{27,3}(C_{27}(2,3,7,11))$. 
\\
$\Rightarrow$ Type-1 set of $C_{27}(1,3,8,10)$ is 

$T1_{27}(C_{27}(1,3,8,10))$ = $\{C_{27}(1,3,8,10), C_{27}(2,6,7,11), C_{27}(4,5,12,13)\}$ 

\hfill = $T1_{27}(C_{27}(2,6,7,11))$ = $T1_{27}(C_{27}(4,5,12,13))$; 
\\
Type-1 group of $C_{27}(1,3,8,10)$ is 

$(T1_{27}(C_{27}(1,3,8,10)), \circ')$ = $(T1_{27}(C_{27}(2,6,7,11)), \circ')$ = $(T1_{27}(C_{27}(4,5,12,13)), \circ')$
\\
where $T1_{27}(C_{27}(1,3,8,10))$ = $\{C_{27}(1,3,8,10), C_{27}(2,6,7,11), C_{27}(4,5,12,13)\}$ 

\hfill = $T1_{27}(C_{27}(2,6,7,11))$ = $T1_{27}(C_{27}(4,5,12,13))$; and
 \\
Type-2 group of $C_{27}(1,3,8,10)$ w.r.t. $m$ = 3 is 

$(T2_{27,3}(C_{27}(1,3,8,10)), \circ)$ = $(T2_{27,3}(C_{27}(3,4,5,13)), \circ)$ = $(T2_{27,3}(C_{27}(2,3,7,11)), \circ)$
\\ 
where $T2_{27,3}(C_{27}(1,3,8,10))$ = $\{C_{27}(1,3,8,10), C_{27}(3,4,5,13), C_{27}(2,3,7,11)\}$ 

\hfill  = $T2_{27,3}(C_{27}(3,4,5,13))$ = $T2_{27,3}(C_{27}(2,3,7,11))$. 

\item [\rm (g)] In problem 5 in \cite{v2-1}, we have

$T1_{48}(C_{48}(1,2,23))$ = $\{C_{48}(1,2,23), C_{48}(5,10,19), C_{48}(7,14,17), C_{48}(11,13,22)\}$ 

\hfill = $T1_{48}(C_{48}(5,10,19))$ = $T1_{48}(C_{48}(7,14,17))$ = $T1_{48}(C_{48}(11,13,22))$  and

$T2_{48,2}(C_{48}(1,2,23))$ = $\{C_{48}(1,2,23), C_{48}(2,11,13)\}$ = $T2_{48,2}(C_{48}(2,11,13))$. 
\\
$\Rightarrow$ Type-1 set of $C_{48}(1,2,23)$ is 

$T1_{48}(C_{48}(1,2,23))$ = $\{C_{48}(1,2,23), C_{48}(5,10,19), C_{48}(7,14,17), C_{48}(11,13,22)\}$ 

\hfill = $T1_{48}(C_{48}(5,10,19))$ = $T1_{48}(C_{48}(7,14,17))$ = $T1_{48}(C_{48}(11,13,22))$; 
\\
Type-1 group of $C_{48}(1,2,23)$ is 

$(T1_{48}(C_{48}(1,2,23)), \circ')$ = $(T1_{48}(C_{48}(5,10,19)), \circ')$ 

\hfill = $(T1_{48}(C_{48}(7,14,17)), \circ')$ = $(T1_{48}(C_{48}(11,13,22)), \circ')$  
\\ 
where $T1_{48}(C_{48}(1,2,23))$ = $\{C_{48}(1,2,23), C_{48}(5,10,19), C_{48}(7,14,17), C_{48}(11,13,22)\}$ 

\hfill = $T1_{48}(C_{48}(5,10,19))$ = $T1_{48}(C_{48}(7,14,17))$ = $T1_{48}(C_{48}(11,13,22))$; and
 \\
Type-2 group of $C_{48}(1,2,23)$ w.r.t. $m$ = 2 is 

$(T2_{48,2}(C_{48}(1,2,23)), \circ)$ = $(T2_{48,2}(C_{48}(2,11,13)), \circ)$ 
\\ 
where $T2_{48,2}(C_{48}(1,2,23))$ = $\{C_{48}(1,2,23), C_{48}(2,11,13)\}$ = $T2_{48,2}(C_{48}(2,11,13))$.

\item [\rm (h)] In problem 5 in \cite{v2-1}, we have

$T1_{48}(C_{48}(1,4,23))$ = $\{C_{48}(1,4,23), C_{48}(5,19,20), C_{48}(7,17,20), C_{48}(4,11,13)\}$ 

\hfill = $T1_{48}(C_{48}(5,19,20))$ = $T1_{48}(C_{48}(7,17,20))$ = $T1_{48}(C_{48}(4,11,13))$  and

$T2_{48,2}(C_{48}(1,4,23))$ = $\{C_{48}(1,4,23), C_{48}(4,11,13)\}$ = $T2_{48,2}(C_{48}(4,11,13))$. 
\\
$\Rightarrow$ Type-1 set of $C_{48}(1,4,23)$ is 

$T1_{48}(C_{48}(1,4,23))$ = $\{C_{48}(1,4,23), C_{48}(5,19,20), C_{48}(7,17,20), C_{48}(4,11,13)\}$ 

\hfill = $T1_{48}(C_{48}(5,19,20))$ = $T1_{48}(C_{48}(7,17,20))$ = $T1_{48}(C_{48}(4,11,13))$; 
\\
Type-1 group of $C_{48}(1,4,23)$ is 

$(T1_{48}(C_{48}(1,4,23)), \circ')$ = $(T1_{48}(C_{48}(5,19,20)), \circ')$ 

\hfill = $(T1_{48}(C_{48}(7,17,20)), \circ')$ = $(T1_{48}(C_{48}(4,11,13)), \circ')$  
\\ 
where $T1_{48}(C_{48}(1,4,23))$ = $\{C_{48}(1,4,23), C_{48}(5,19,20), C_{48}(7,17,20), C_{48}(4,11,13)\}$ 

\hfill = $T1_{48}(C_{48}(5,19,20))$ = $T1_{48}(C_{48}(7,17,20))$ = $T1_{48}(C_{48}(4,11,13))$; and
 \\
Type-2 group of $C_{48}(1,4,23)$ w.r.t. $m$ = 2 is 

$(T2_{48,2}(C_{48}(1,4,23)), \circ)$ = $(T2_{48,2}(C_{48}(4,11,13)), \circ)$ 
\\ 
where $T2_{48,2}(C_{48}(1,4,23))$ = $\{C_{48}(1,4,23), C_{48}(4,11,13)\}$ = $T2_{48,2}(C_{48}(4,11,13))$.

\item [\rm (i)] In problem 5 in \cite{v2-1}, we have

$T1_{48}(C_{48}(1,6,23))$ = $\{C_{48}(1,6,23), C_{48}(5,18,19), C_{48}(6,7,17), C_{48}(11,13,18)\}$ 

\hfill = $T1_{48}(C_{48}(5,18,19))$ = $T1_{48}(C_{48}(6,7,17))$ = $T1_{48}(C_{48}(11,13,18))$  and

$T2_{48,2}(C_{48}(1,6,23))$ = $\{C_{48}(1,6,23), C_{48}(6,11,13)\}$ = $T2_{48,2}(C_{48}(6,11,13))$. 
\\
$\Rightarrow$ Type-1 set of $C_{48}(1,6,23)$ is 

$T1_{48}(C_{48}(1,6,23))$ = $\{C_{48}(1,6,23), C_{48}(5,18,19), C_{48}(6,7,17), C_{48}(11,13,18)\}$ 

\hfill = $T1_{48}(C_{48}(5,18,19))$ = $T1_{48}(C_{48}(6,7,17))$ = $T1_{48}(C_{48}(11,13,18))$; 
\\
Type-1 group of $C_{48}(1,6,23)$ is 

$(T1_{48}(C_{48}(1,6,23)), \circ')$ = $(T1_{48}(C_{48}(5,18,19)), \circ')$ = $(T1_{48}(C_{48}(6,7,17)), \circ')$ 

\hfill = $(T1_{48}(C_{48}(11,13,18)), \circ')$   
\\
where $T1_{48}(C_{48}(1,6,23))$ = $\{C_{48}(1,6,23), C_{48}(5,18,19), C_{48}(6,7,17), C_{48}(11,13,18)\}$ 

\hfill = $T1_{48}(C_{48}(5,18,19))$ = $T1_{48}(C_{48}(6,7,17))$ = $T1_{48}(C_{48}(11,13,18))$; and
 \\
Type-2 group of $C_{48}(1,6,23)$ w.r.t. $m$ = 2 is 

$(T2_{48,2}(C_{48}(1,6,23)), \circ)$ = $(T2_{48,2}(C_{48}(6,11,13)), \circ)$ 
\\ where $T2_{48,2}(C_{48}(1,6,23))$ = $\{C_{48}(1,6,23), C_{48}(6,11,13)\}$ = $T2_{48,2}(C_{48}(6,11,13))$.

\item [\rm (j)] In problem 6 in \cite{v2-1}, we have
\\
$T1_{54}(C_{54}(1,3,17,19))$ = $\{C_{54}(1,3,17,19), C_{54}(5,13,15,23), C_{54}(7,11,21,25)\}$  

\hfill = $T1_{54}(C_{54}(5,13,15,23))$ =  $T1_{54}(C_{54}(7,11,21,25))$ and
\\
$T2_{54,3}(C_{54}(1,3,17,19))$ = $\{C_{54}(1,3,17,19), C_{54}(3,7,11,25), C_{54}(3,5,13,23)\}$ 

\hfill  = $T2_{54,3}(C_{54}(3,7,11,25))$ = $T2_{54,3}(C_{54}(3,5,13,23))$. 
\\
$\Rightarrow$ Type-1 set of $C_{54}(1,3,17,19)$ is 

$T1_{54}(C_{54}(1,3,17,19))$ = $\{C_{54}(1,3,17,19), C_{54}(5,13,15,23), C_{54}(7,11,21,25)\}$  

\hfill = $T1_{54}(C_{54}(5,13,15,23))$ =  $T1_{54}(C_{54}(7,11,21,25))$; 
\\
Type-1 group of $C_{54}(1,3,17,19)$ is 

$(T1_{54}(C_{54}(1,3,17,19)), \circ')$ = $(T1_{54}(C_{54}(5,13,15,23)), \circ')$ = $(T1_{54}(C_{54}(7,11,21,25)), \circ')$
\\
where $T1_{54}(C_{54}(1,3,17,19))$ = $\{C_{54}(1,3,17,19), C_{54}(5,13,15,23), C_{54}(7,11,21,25)\}$  

\hfill = $T1_{54}(C_{54}(5,13,15,23))$ =  $T1_{54}(C_{54}(7,11,21,25))$; and
 \\
Type-2 group of $C_{54}(1,3,17,19)$ w.r.t. $m$ = 3 is 

$(T2_{54,3}(C_{54}(1,3,17,19)), \circ)$ = $(T2_{54,3}(C_{54}(3,7,11,25)), \circ)$ = $(T2_{54,3}(C_{54}(3,5,13,23)), \circ)$
\\ 
where $T2_{54,3}(C_{54}(1,3,17,19))$ = $\{C_{54}(1,3,17,19), C_{54}(3,7,11,25), C_{54}(3,5,13,23)\}$ 

\hfill = $T2_{54,3}(C_{54}(3,7,11,25))$ = $T2_{54,3}(C_{54}(3,5,13,23))$. 

\item [\rm (k)] In problem 6 in \cite{v2-1}, we have

$T1_{54}(C_{54}(1,6,17,19))$ = $\{C_{54}(1,6,17,19), C_{54}(5,13,23,24), C_{54}(7,11,12,25)\}$ 

\hfill = $T1_{54}(C_{54}((5,13,23,24))$ =  $T1_{54}(C_{54}(7,11,12,25))$ and

$T2_{54,3}(C_{54}(1,6,17,19))$ = $\{C_{54}(1,6,17,19), C_{54}(6,7,11,25), C_{54}(5,6,13,23)\}$ 

\hfill = $T2_{54,3}(C_{54}(6,7,11,25))$ = $T2_{54,3}(C_{54}(5,6,13,23))$. 
\\
$\Rightarrow$ Type-1 set of $C_{54}(1,6,17,19)$ is 

$T1_{54}(C_{54}(1,6,17,19))$ = $\{C_{54}(1,6,17,19), C_{54}(5,13,23,24), C_{54}(7,11,12,25)\}$ 

\hfill = $T1_{54}(C_{54}((5,13,23,24))$ =  $T1_{54}(C_{54}(7,11,12,25))$; 
\\
Type-1 group of $C_{54}(1,6,17,19)$ is 

$(T1_{54}(C_{54}(1,6,17,19)), \circ')$ = $(T1_{54}(C_{54}(5,13,23,24)), \circ')$ = $(T1_{54}(C_{54}(7,11,12,25)), \circ')$
\\
 where $T1_{54}(C_{54}(1,6,17,19))$ = $\{C_{54}(1,6,17,19), C_{54}(5,13,23,24), C_{54}(7,11,12,25)\}$ 

\hfill = $T1_{54}(C_{54}((5,13,23,24))$ =  $T1_{54}(C_{54}(7,11,12,25))$; and
 \\
Type-2 group of $C_{54}(1,6,17,19)$ w.r.t. $m$ = 3 is 

$(T2_{54,3}(C_{54}(1,6,17,19)), \circ)$ = $(T2_{54,3}(C_{54}(6,7,11,25)), \circ)$ = $(T2_{54,3}(C_{54}(5,6,13,23)), \circ)$
\\
 where $T2_{54,3}(C_{54}(1,6,17,19))$ = $\{C_{54}(1,6,17,19), C_{54}(6,7,11,25), C_{54}(5,6,13,23)\}$ 

\hfill = $T2_{54,3}(C_{54}(6,7,11,25))$ = $T2_{54,3}(C_{54}(5,6,13,23))$. 

\item [\rm (l)] In problem 6 in \cite{v2-1}, we have
\\
$T1_{54}(C_{54}(1,17,18,19))$ = $\{C_{54}(1,17,18,19), C_{54}(5,13,18,23), C_{54}(7,11,18,25)\}$  

\hfill = $T1_{54}(C_{54}(5,13,18,23))$ =  $T1_{54}(C_{54}(7,11,18,25))$ and
\\
$T2_{54,3}(C_{54}(1,17,18,19))$ = $\{C_{54}(1,3,17,19)\}$. 
\\
$\Rightarrow$ Type-1 set of $C_{54}(1,17,18,19)$ is 

$T1_{54}(C_{54}(1,17,18,19))$ = $\{C_{54}(1,3,17,19), C_{54}(5,13,18,23), C_{54}(7,11,18,25)\}$  

\hfill = $T1_{54}(C_{54}(5,13,18,23))$ =  $T1_{54}(C_{54}(7,11,18,25))$; 
\\
Type-1 group of $C_{54}(1,17,18,19)$ is 

$(T1_{54}(C_{54}(1,17,18,19)), \circ')$ = $(T1_{54}(C_{54}(5,13,18,23)), \circ')$ = $(T1_{54}(C_{54}(7,11,18,25)), \circ')$
\\
 where $T1_{54}(C_{54}(1,17,18,19))$ = $\{C_{54}(1,3,17,19), C_{54}(5,13,18,23), C_{54}(7,11,18,25)\}$  

\hfill = $T1_{54}(C_{54}(5,13,18,23))$ =  $T1_{54}(C_{54}(7,11,18,25))$; and
 \\
Type-2 group of $C_{54}(1,17,18,19)$ w.r.t. $m$ = 3 is 

$(T2_{54,3}(C_{54}(1,17,18,19)), \circ)$ where $T2_{54,3}(C_{54}(1,17,18,19))$ = $\{C_{54}(1,17,18,19)\}$.

\item [\rm (m)] In problem 7 in \cite{v2-1}, we have
\\
$T1_{108}(C_{108}(3,5,31,41))$ = $\{C_{108}(3,5,31,41), C_{108}(11,15,25,47), C_{108}(1,21,35,37)$, 

\hfill $C_{108}(17,19,33,53)$,  $ C_{108}(7,29,39,43), C_{108}(13,23,49,51)\}$ 

\hfill = $T1_{108}(C_{108}(11,15,25,47))$  =  $T1_{108}(C_{108}(1,21,35,37))$ =  $T1_{108}(C_{108}(17,19,33,53))$ 

\hfill  =  $T1_{108}(C_{108}(7,29,39,43))$ =  $T1_{108}(C_{108}(13,23,49,51))$ and
\\
$T2_{108,3}(C_{108}(3,5,31,41))$ = $\{C_{108}(3,5,31,41), C_{108}(3,7,29,43),  C_{108}(3,17,19,53)\}$ 

\hfill = $T2_{108,3}(C_{108}(3,7,29,43))$ = $T2_{108,3}(C_{108}(3,17,19,53))$. 
\\
$\Rightarrow$ Type-1 set of $C_{108}(3,5,31,41)$ is 

$T1_{108}(C_{108}(3,5,31,41))$ = $\{C_{108}(3,5,31,41), C_{108}(11,15,25,47), C_{108}(1,21,35,37)$, 

\hfill $C_{108}(17,19,33,53)$, $ C_{108}(7,29,39,43), C_{108}(13,23,49,51)\}$ 

\hfill = $T1_{108}(C_{108}(11,15,25,47))$  =  $T1_{108}(C_{108}(1,21,35,37))$ =  $T1_{108}(C_{108}(17,19,33,53))$

\hfill   =  $T1_{108}(C_{108}(7,29,39,43))$ =  $T1_{108}(C_{108}(13,23,49,51))$; 
\\
Type-1 group of $C_{108}(3,5,31,41)$ is 

$(T1_{108}(C_{108}(3,5,31,41)), \circ')$ = $(T1_{108}(C_{108}(11,15,25,47)), \circ')$ 

\hfill = $(T1_{108}(C_{108}(1,21,35,37)), \circ')$ = $(T1_{108}(C_{108}(17,19,33,53)), \circ')$ 

\hfill = $(T1_{108}(C_{108}(7,29,39,43)), \circ')$ = $(T1_{108}(C_{108}(13,23,49,51)), \circ')$ where
\\
$T1_{108}(C_{108}(3,5,31,41))$ = $\{C_{108}(3,5,31,41), C_{108}(11,15,25,47), C_{108}(1,21,35,37)$, 

\hfill $C_{108}(17,19,33,53)$, $ C_{108}(7,29,39,43), C_{108}(13,23,49,51)\}$ 

\hfill = $T1_{108}(C_{108}(11,15,25,47))$  =  $T1_{108}(C_{108}(1,21,35,37))$ =  $T1_{108}(C_{108}(17,19,33,53))$ 

\hfill  =  $T1_{108}(C_{108}(7,29,39,43))$ =  $T1_{108}(C_{108}(13,23,49,51))$; and
 \\
Type-2 group of $C_{108}(3,5,31,41)$ w.r.t. $m$ = 3 is 
\\
$(T2_{108,3}(C_{108}(3,5,31,41)), \circ)$ = $(T2_{108,3}(C_{108}(3,7,29,43)), \circ)$ = $(T2_{108,3}(C_{108}(3,17,19,53)), \circ)$ 
\\
where $T2_{108,3}(C_{108}(3,5,31,41))$ = $\{C_{108}(3,5,31,41), C_{108}(3,7,29,43),  C_{108}(3,17,19,53)\}$ 

\hfill = $T2_{108,3}(C_{108}(3,7,29,43))$ = $T2_{108,3}(C_{108}(3,17,19,53))$.

\item [\rm (n)] In problem 7 in \cite{v2-1}, we have
\\
$T1_{108}(C_{108}(5,12,31,41))$ = $\{C_{108}(5,12,31,41), C_{108}(11,25,47,48), C_{108}(1,24,35,37)$, 

\hfill $C_{108}(17,19,24,53), C_{108}(7,29,43,48), C_{108}(12,13,23,49)\}$ 

\hfill = $T1_{108}(C_{108}(11,25,47,48))$ =  $T1_{108}(C_{108}(1,24,35,37))$ =  $T1_{108}(C_{108}(17,19,24,53))$ 

\hfill =  $T1_{108}(C_{108}(7,29,43,48))$ =  $T1_{108}(C_{108}(12,13,23,49))$ and
\\
$T2_{108,3}(C_{108}(5,12,31,41))$ = $\{C_{108}(5,12,31,41), C_{108}(7,12,29,43),  C_{108}(12,17,19,53)\}$ 

\hfill = $T2_{108,3}(C_{108}(7,12,29,43))$ = $T2_{108,3}(C_{108}(12,17,19,53))$. 
\\
$\Rightarrow$ Type-1 set of $C_{108}(5,12,31,41)$ is 

$T1_{108}(C_{108}(1,17,18,19))$ = $\{C_{108}(1,3,17,19), C_{108}(5,13,18,23), C_{108}(7,11,18,25)\}$  

\hfill = $T1_{108}(C_{108}(5,13,18,23))$ =  $T1_{108}(C_{108}(7,11,18,25))$; 
\\
Type-1 group of $C_{108}(5,12,31,41)$ is 

$(T1_{108}(C_{108}(5,12,31,41)), \circ')$ = $(T1_{108}(C_{108}(11,25,47,48)), \circ')$ 

\hfill = $(T1_{108}(C_{108}(1,24,35,37)), \circ')$ = $(T1_{108}(C_{108}(17,19,24,53)), \circ')$ 

\hfill = $(T1_{108}(C_{108}(7,29,43,48)), \circ')$ = $(T1_{108}(C_{108}(12,13,23,49)), \circ')$

\hfill where $T1_{108}(C_{108}(5,12,31,41))$ = $\{C_{108}(5,12,31,41), C_{108}(11,25,47,48), C_{108}(1,24,35,37)$, 

\hfill $C_{108}(17,19,24,53), C_{108}(7,29,43,48), C_{108}(12,13,23,49)\}$ 

\hfill = $T1_{108}(C_{108}(11,25,47,48))$ =  $T1_{108}(C_{108}(1,24,35,37))$ =  $T1_{108}(C_{108}(17,19,24,53))$ 

\hfill =  $T1_{108}(C_{108}(7,29,43,48))$ =  $T1_{108}(C_{108}(12,13,23,49))$; and
 \\
Type-2 group of $C_{108}(5,12,31,41)$ w.r.t. $m$ = 3 is 

$(T2_{108,3}(C_{108}(5,12,31,41)), \circ)$ =  $(T2_{108,3}(C_{108}(7,12,29,43)), \circ)$ 

\hfill = $(T2_{108,3}(C_{108}(12,17,19,53)), \circ)$
\\
where $T2_{108,3}(C_{108}(5,12,31,41))$ = $\{C_{108}(5,12,31,41), C_{108}(7,12,29,43),  C_{108}(12,17,19,53)\}$ 

\hfill = $T2_{108,3}(C_{108}(7,12,29,43))$ = $T2_{108,3}(C_{108}(12,17,19,53))$. 

\item [\rm (o)] In problem 7 in \cite{v2-1}, we have
\\
$T1_{108}(C_{108}(5,18,31,41))$ = $\{C_{108}(5,18,31,41), C_{108}(11,18,25,47), C_{108}(1,18,35,37)$, 

\hfill $C_{108}(17,18,19,53), C_{108}(7,18,29,43), C_{108}(13,18,23,49)\}$ 

\hfill = $T1_{108}(C_{108}(11,18,25,47))$ =  $T1_{108}(C_{108}(1,18,35,37))$ =  $T1_{108}(C_{108}(17,18,19,53))$ 

\hfill =  $T1_{108}(C_{108}(7,18,29,43))$ =  $T1_{108}(C_{108}(13,18,23,49))$  and
\\
$T2_{108,3}(C_{108}(5,18,31,41))$ = $\{C_{108}(5,18,31,41)\}$. 
\\
$\Rightarrow$ Type-1 set of $C_{108}(5,18,31,41)$ is 

$T1_{108}(C_{108}(5,18,31,41))$ = $\{C_{108}(5,18,31,41), C_{108}(11,18,25,47), C_{108}(1,18,35,37)$, 

\hfill $C_{108}(17,18,19,53), C_{108}(7,18,29,43), C_{108}(13,18,23,49)\}$ 

\hfill = $T1_{108}(C_{108}(11,18,25,47))$ =  $T1_{108}(C_{108}(1,18,35,37))$ =  $T1_{108}(C_{108}(17,18,19,53))$ 

\hfill =  $T1_{108}(C_{108}(7,18,29,43))$ =  $T1_{108}(C_{108}(13,18,23,49))$; and 
\\
Type-1 group of $C_{108}(5,18,31,41)$ is 

$(T1_{108}(C_{108}(5,18,31,41)), \circ')$ = $(T1_{108}(C_{108}(11,18,25,47)), \circ')$ 

 = $(T1_{108}(C_{108}(1,18,35,37)), \circ')$ = $(T1_{108}(C_{108}(17,18,19,53)), \circ')$ 

 = $(T1_{108}(C_{108}(7,18,29,43)), \circ')$ = $(T1_{108}(C_{108}(13,18,23,49)), \circ')$
\\
where $T1_{108}(C_{108}(5,18,31,41))$ = $\{C_{108}(5,18,31,41), C_{108}(11,18,25,47), C_{108}(1,18,35,37)$, 

\hfill $C_{108}(17,18,19,53), C_{108}(7,18,29,43), C_{108}(13,18,23,49)\}$ 

\hfill = $T1_{108}(C_{108}(11,18,25,47))$ =  $T1_{108}(C_{108}(1,18,35,37))$ =  $T1_{108}(C_{108}(17,18,19,53))$ 

\hfill =  $T1_{108}(C_{108}(7,18,29,43))$ =  $T1_{108}(C_{108}(13,18,23,49))$; and
 \\
Type-2 group of $C_{108}(5,18,31,41)$ w.r.t. $m$ = 3 is 

$(T2_{108,3}(C_{108}(5,18,31,41)), \circ)$ where $T2_{108,3}(C_{108}(5,18,31,41))$ = $\{C_{108}(5,18,31,41)\}$.  \hfill $\Box$
\end{enumerate}

\begin{prm} \label{p4} {\rm  For $i$ = 1 to 7, 
\\
(a) find $T2_{1715,7}(C_{1715}(G_i))$ and
\\
(b) prove that $(T2_{1715,7}(C_{1715}(G_i)), \circ)$ is an Abelian group where 

$G_1$ = $\{7,17,228,262,473,507,718,752\}$,

$G_2$ = $\{7,122,123,367,368,612,613,857\}$,  

$G_3$ = $\{7,18,227,263,472,508,717,753\}$,

$G_4$ = $\{7,87,158,332,403,577,648,822\}$,

$G_5$ = $\{7,53,192,298,437,543,682,788\}$,

$G_6$ = $\{7,52,193,297,438,542,683,787\}$ and

$G_7$ = $\{7,88,157,333,402,578,647,823\}$.}  
\end{prm}
\noindent
{\bf Solution} Here, 1715 = $5\times {7}^3$ = $np^3$ and all the 7 sets $G_1$, $G_2$, . . . , $G_7$ are subsets of $\mathbb{Z}_{1715/2}$ and have 7 as a common element where $n$ = 5, $p$ = 7 and $np^2$ = 245. Thus, there is a possibility that $C_n(G_i)$ may be of Type-2 isomorphic circulant graphs of order 1715 = $5\times {7}^3$ w.r.t.  $m$ = 7, $1 \leq i \leq 7$. Using Theorem \ref{t33}, we show that circulant graphs $C_{1715}(G_i)$ are isomorphic and of Type-2 w.r.t.  $m$ = 7 = $p$,  $1 \leq i \leq p$ = 7.  

The minimum jump size, other than 7, in the seven $G_i$ sets is 17 which implies, 17 = $x+yp$ = $3 + 2\times 7$ which implies, $x = 3$ and $y = 2$. See Theorem \ref{t33}. Thus, 

$d^{np^3,x+yp}_i$ = $d^{5\times {7}^3,3+2\times 7}_i$ = $(i-1)xpn+yp+x$ = $105(i-1)+17$, $i = $ 1 to 7. This implies, 

$d^{5\times {7}^3,3+2\times 7}_1$ = $17\in G_1$; $d^{5\times {7}^3,3+2\times 7}_2$ = $122\in G_2$; $d^{5\times {7}^3,3+2\times 7}_3$ = $227\in G_3$; 

$d^{5\times {7}^3,3+2\times 7}_4$ = $332\in G_4$; $d^{5\times {7}^3,3+2\times 7}_5$ = $437\in G_5$; $d^{5\times {7}^3,3+2\times 7}_6$ = $542\in G_6$ and

$d^{5\times {7}^3,3+2\times 7}_7$ = $647\in G_7$.
\\
For $i =$ 1 to $p$, we have 

$R^{np^3,yp+x}_i$ = $\{p,$ $d^{np^3,yp+x}_i$, $np^2-d^{np^3,yp+x}_i$, $np^2+d^{np^3,yp+x}_i$, $2np^2-d^{np^3,yp+x}_i$, $2np^2+d^{np^3,yp+x}_i$,

\hfill . . . , $(p-1)np^2-d^{np^3,yp+x}_i$, $(p-1)np^2+d^{np^3,yp+x}_i$, $np^3-d^{np^3,yp+x}_i$, $np^3-p\}$. 
\\
This implies, under $modulo$ 1715,

$R^{5\times 7^3,2\times 7+3}_1$ = $\{7, 17, 228, 262, 473, 507, 718, 752, $ 

\hfill $963, 997, 1208, 1242, 1453, 1487, 1698, 1708\}$ = $G_1 \cup (1715 - G_1)$; 

$R^{5\times 7^3,2\times 7+3}_2$ = $\{7, 122, 123, 367, 368, 612, 613, 857, $ 

\hfill $858, 1102, 1103, 1347, 1348, 1592, 1593, 1708\}$ = $G_2 \cup (1715 - G_2)$; 

$R^{5\times 7^3,2\times 7+3}_3$ = $\{7, 227, 18, 472, 263, 717, 508, 962, 753, 1207, 998,$

\hfill $1452, 1243, 1697, 1488, 1708\}$ = $G_3 \cup (1715 - G_3)$; 

$R^{5\times 7^3,2\times 7+3}_4$ = $\{7, 332, (245-332)+1715 = 1628, 577, 158, 822, 403, 1067,$ 

\hfill $648, 1312, 893, 1557, 1138,  1802 = 87, 1383, 1708\}$ = $G_4 \cup (1715 - G_4)$; 

$R^{5\times 7^3,2\times 7+3}_5$ = $\{7, 437, (245-437)+1715 = 1523, 682, 53, 927,$ 

\hfill  $298, 1172, 543, 1417, 788, 1662, 1033, 192, 1278, 1708\}$ = $G_5 \cup (1715 - G_5)$; 

$R^{5\times 7^3,2\times 7+3}_6$ = $\{7, 542, (245-542)+1715 = 1418, 787, 1663, 1032,$ 

\hfill $ 193, 1277, 438, 1522, 683, 52, 928, 297, 1173, 1708\}$ = $G_6 \cup (1715 - G_6)$;

$R^{5\times 7^3,2\times 7+3}_7$ = $\{7, 647, (245-647)+1715 = 1313, 892, 1558, 1137,$ 

\hfill $88, 1382, 333, 1627, 578, 157, 823, 402, 1068, 647, 1708\}$ = $G_7 \cup (1715 - G_7)$. 

\vspace{.2cm}
\noindent
{\bf{\it Claim 1.}} $\Theta_{1715,7,n}(R^{1715,17}_1)$ = $R^{1715,17}_2$, $n$ = 5.   

\vspace{.2cm}
\noindent
$\Theta_{1715,7,n}(R^{1715,17}_1)$ = $\Theta_{1715,7,5}(R^{1715,17}_1)$ = $\Theta_{1715,7,5}(\{7, 17, 228, 262, 473,$ 

~\hfill $507, 718, 752, 963, 997, 1208, 1242, 1453, 1487, 1698, 1708 \})$ 

 = $\Theta_{1715,7,5}(\{7, 1708\})$ $\bigcup$ $\Theta_{1715,7,5}(\{17, 262, 507, 752, 997, 1242, 1487\})$ 

~\hfill $\bigcup$ $\Theta_{1715,7,5}(\{228, 473, 718, 963, 1208, 1453, 1698\})$

= $\{7, 1708\} \bigcup (3\times 7 \times 5+\{17, 262, 507, 752, 997, 1242, 1487\})$

~\hfill $\bigcup$ $(4\times 7 \times 5+\{228, 473, 718, 963, 1208, 1453, 1698\})$ 

 = $\{7, 1708\}$ $\bigcup$ $\{122, 367, 612, 857, 1102, 1347, 1592\} 
\bigcup \{368, 613, 858, 1103, 1348, 1593, 123\}$ = $R^{1715,17}_2$.

Hence Claim 1 is true.

\vspace{.2cm}
\noindent
{\it Claim 2.} $\Theta_{1715,7,4n}(R^{1715,17}_1)$ = $R^{1715,17}_5$, $n = 5$.   

\vspace{.2cm}
\noindent
 $\Theta_{1715,7,4n}(R^{1715,17}_1)$ = $\Theta_{1715,7,20}(\{7, 1708\}) \bigcup \Theta_{1715,7,20}(\{17, 262, 507, 752, 997, 1242, 1487\})$ 

~\hfill $\bigcup$ $\Theta_{1715,7,20}(\{228, 473, 718, 963, 1208, 1453, 1698\})$

 = $\{7, 1708\}$ $\bigcup$ $(3\times 20 \times 7+\{17, 262, 507, 752, 997, 1242, 1487\})$

~\hfill $\bigcup$ $(4\times 20 \times 7+\{228, 473, 718, 963, 1208, 1453, 1698\})$ 

 = $\{7, 1708\}$ $\bigcup$ $\{437, 682, 927, 1172, 1417, 1662, 1907 = 192\}$ 

~\hfill $\bigcup$ $\{788, 1033, 1278, 1523, 1768 = 53, 2013 = 298, 543\}$ = $R^{1715,17}_5$.

Hence Claim 2 is true.

$\Rightarrow$ $\Theta_{1715,7,n}(C_{1715}(G_1))$ = $\Theta_{1715,7,n}(C_{1715}(R^{1715,17}_1))$ = $C_{1715}(R^{1715,17}_2)$ = $C_{1715}(G_2)$ and 

\hspace{.25cm} $\Theta_{1715,7,4n}(C_{1715}(G_1))$ = $\Theta_{1715,7,4n}(C_{1715}(R^{1715,17}_1))$ = $C_{1715}(R^{1715,17}_5)$ = $C_{1715}(G_5)$.

$\Rightarrow$ $C_{1715}(G_1)$, $C_{1715}(G_2)$ and $C_{1715}(G_5)$ are isomorphic graphs.

Similarly, we can show that $\Theta_{1715,7,jn}(C_{1715}(G_1))$ = $C_{1715}(G_{j+1})$ as well as $\Theta_{1715,7,jn}(C_{1715}(G_i))$ = $C_{1715}(G_{i+j})$ where $i+j$ in $G_{i+j}$ is calculated under addition modulo 7, $0 \leq i,j \leq 6$ and $G_0$ = $G_7$. This implies, $C_{1715}(G_i)$ are isomorphic circulant graphs for $i$ = 1 to 7.

Also, for $1 \leq i,j \leq 7$, $i \neq j$, $1 < s < 6$ and $s\in\varphi_{7}$, $sR^{1715,17}_i$ $\neq$ $R^{1715,17}_i, R^{1715,17}_j$ since $7\in R^{1715,17}_i$ but $7\notin sR^{1715,17}_i$ whereas $7s\in sR^{1715,17}_i$ but $7s \notin R^{1715,17}_i$. This implies, $C_{1715}(R^{1715,17}_i)$ and $C_{1715}(R^{1715,17}_j)$ are not Adam's isomorphic when $i \neq j$ and $1 \leq i,j \leq 7$. This implies, $C_{1715}(G_i)$ and $C_{1715}(G_{j})$ are isomorphic but not of Type-1 isomorphic when $i \neq j$ and $1 \leq i,j \leq 7$. 

This implies, $\Theta_{1715,7,jn}(C_{1715}(G_1))$ = $\Theta_{1715,7,jn}(C_{1715}(R^{1715,17}_1))$ = $C_{1715}(R^{1715,17}_{j+1})$ = $C_{1715}(G_{j+1})$ are isomorphic circulant graphs of Type-2 w.r.t.  $r = 7$, $0 \leq j \leq p-1 = 6$.  

Then the result follows from Theorem \ref{t33}. \hfill $\Box$

\section{Conclusion}

In Theorem \ref{t33}, we have $T2_{n,r}(C_{n}(R)) \subseteq V_{n,r}(C_{n}(R))$ for any circulant graph $C_{n}(R)$. It is observed that $T2_{54,3}(C_{54}(2,3,16,20)) \subsetneq V_{54,3}(C_{54}(2,3,16,20))$ in problem \ref{p1} and $T2_{54,3}(C_{54}(3,7,20,34)) \subsetneq V_{54,3}(C_{54}(3,7,20,34))$ in problem \ref{p2}. This need not be the case always. In problem 4 in \cite{v2-1}, we have $\theta_{27,3,1}(C_{27}(1,3,8,10))$ = $C_{27}(3,4,5,13)$ and $\theta_{27,3,2}(C_{27}(1,3,8,10))$ = $C_{27}(2,3,7,11)$ and thereby, $T2_{27,3}(C_{27}(1,3, ~8,10))$ = $\{C_{27}(1,3,8,10)$, $C_{27}(3,4,5,13)$, $C_{27}(2,3,7,11)\}$ = $V_{27,3}(C_{27}(1,3, ~8,10))$ = $\{\theta_{27,3,t}(C_{27}(1,3,8,10)):$ $0 \leq t \leq \frac{27}{3}-1 = 8\}$ = $\{\theta_{27,3,t}(C_{27}(1,3,8,10)):$ $0 \leq t \leq 2\}$. Based on the above observation, we propose the following open problems related to  $T2_{n,m}(C_{n}(R))$ and $V_{n,m}(C_{n}(R))$ for further research.

\begin{oprm}\quad {\rm \label{op1} Characterise and find all circulant graphs $C_n(R)$ such that  

(i) $T2_{n,r}(C_{n}(R))$ = $\{C_{n}(R)\}$. $i.e.$, $|T2_{n,r}(C_{n}(R))|$ = 1;  

(ii)  $T2_{n,r}(C_{n}(R))$ $\neq$ $\{C_{n}(R)\}$. $i.e.$,  $|T2_{n,r}(C_{n}(R))| > 1$. \hfill $\Box$}
\end{oprm}

\begin{oprm}\quad {\rm \label{op2} Characterise and find all circulant graphs $C_n(R)$ such that  

(i) $T2_{n,r}(C_{n}(R))$ = $V_{n,r}(C_{n}(R))$;  

(ii)  $T2_{n,r}(C_{n}(R))$ $\neq$ $V_{n,r}(C_{n}(R))$. $i.e.$,  $T2_{n,r}(C_{n}(R)) \subset V_{n,r}(C_{n}(R))$. \hfill $\Box$}
\end{oprm}

\vspace{.1cm}
\noindent
\textbf{Declaration of competing interest}\quad 
The authors declare that they have no conflict of interest.


\end{document}